\documentclass[a4paper, 11pt]{amsart}

\usepackage{latexsym,verbatim,ifthen,graphicx,mathrsfs}

\usepackage[english]{babel}
\usepackage[utf8]{inputenc} 		
\usepackage{amssymb}
\usepackage{amsmath,tabularx,array,makecell}
\usepackage{amsthm}		
\usepackage{quiver}		
\usepackage{tikz-cd}				
\usetikzlibrary{matrix,arrows,decorations.pathmorphing}
\usepackage[all]{xy}				
\usepackage[colorlinks=true]{hyperref}  
\usepackage{anysize}						
\marginsize{2.5cm}{2.5cm}{2cm}{2cm}
\usepackage{booktabs}			
\usepackage{enumerate}			
\usepackage{multicol}	
\usepackage{xcolor}
\usepackage[nameinlink,capitalise,noabbrev]{cleveref}
\usepackage{microtype}

\usepackage{calc}

\usepackage{mathrsfs}									
\usepackage{fourier}				
\usepackage{fancyhdr}

\newtheorem{theorem}{Theorem}[section]

\newtheorem{lemma}[theorem]{Lemma}
\newtheorem{proposition}[theorem]{Proposition}

\newtheoremstyle{TheoremNum}
{}{}              %
{\itshape}                      %
{}                              %
{\bfseries}                     %
{.}                             %
{ }                             %
{\thmname{#1}\thmnote{ \bfseries #3}}%
\theoremstyle{TheoremNum}
\newtheorem{thmn}{Theorem}
\newtheorem{lemn}{Lemma}

\theoremstyle{definition} %
\newtheorem{example}[theorem]{Example}
\newtheorem{definition}[theorem]{Definition}
\newtheorem{remark}[theorem]{Remark}
\newtheorem{remarks}[theorem]{Remarks}

\DeclareMathOperator{\Fun}{Fun}

\DeclareMathOperator{\Aut}{Aut}
\DeclareMathOperator{\Hom}{Hom}

\DeclareMathOperator{\id}{id}

\newcommand{\APL}{A_{\mathrm{PL}}}
\newcommand{\Ho}{\mathrm{Ho}}

\newcommand{\Z}{{\mathbb{Z}}}
\newcommand{\Q}{{\mathbb{Q}}}

\newcommand{\C}{{\mathcal{C}}}
\newcommand{\E}{\mathcal{E}}
\newcommand{\D}{{\mathcal{D}}}

\newcommand{\op}{\textup{op}}
\newcommand{\inj}{\textup{inj}}
\newcommand{\proj}{\textup{proj}}
\newcommand{\sSet}{\mathsf{sSet}}
\newcommand{\cDGL}{\mathsf{cDGL}}

\newcommand{\Vect}{\mathsf{Vect}}

\newcommand{\CDGA}{\mathsf{CDGA}}

\newcommand{\Ch}{\mathrm{Ch}}
\newcommand{\Chz}{\mathrm{Ch}^{\geq 0}}
\newcommand{\Chun}{\mathrm{Ch}^{\geq 1}}
\newcommand{\Chaz}{\mathrm{Ch}_{\geq 0}}

\newcommand{\GsSet}{G\textrm{-}\mathsf{sSet}}

\newcommand{\OG}{{\mathcal{O}_G}}

\newcommand{\OGop}{{\mathcal{O}^\op_G}}

\DeclareMathOperator{\ev}{ev}
\newcommand{\QWH}{\Q[W(H)]\textrm{-}\mathsf{Mod}}
\newcommand{\V}{{\mathcal{V}}}
\newcommand{\A}{\mathcal{A}}

\renewcommand{\S}{\mathcal{S}}
\newcommand{\I}{\mathcal{I}}
\newcommand{\J}{\mathcal{J}}
\renewcommand{\O}{\mathcal{O}}
\newcommand{\PH}{\underline{P}_H}
\newcommand{\FW}{\mathcal{F}\mathcal{W}}
\newcommand{\Sub}{\mathrm{Sub}}

\usepackage{stmaryrd} %

\definecolor{dark-red}{rgb}{0.6,.15,.15}
\definecolor{dark-blue}{rgb}{.1,.1,.65}
\definecolor{dark-green}{rgb}{.1,.65,.1}
\hypersetup{pdfauthor={José Manuel Moreno Fernández, Bruno Stonek},colorlinks=true, citecolor=dark-blue, urlcolor=dark-green, linkcolor=dark-red,breaklinks=true,hypertexnames=false}

\author{José Manuel Moreno Fernández}
\address{Departamento de Álgebra, Geometría y Topología, Universidad de Málaga, 29080, Málaga, Spain}
\email{josemoreno@uma.es}

\author{Bruno Stonek}
\address{Faculty of Mathematics, Informatics and Mechanics, University of Warsaw, ul. Banacha 2, 02-097 Warszawa, Poland}
\email{bstonek@mimuw.edu.pl}

\begin{document}

\title{Algebraic models for equivariant rational homotopy theory for discrete groups}
\date{}

\begin{abstract} We provide a framework which generalizes algebraic models of a homotopy theory of spaces to 
the genuine equivariant case for a discrete group. 
We explain how this applies to commutative differential graded algebra (cdga) models and complete differential graded Lie algebra models for rational spaces.
We compare the cdga model to other model categories in the literature, correcting some mistakes in previous approaches.

As an intermediary result, we get that the injective model structure on the category of non-negative cochain complexes with values in a Grothendieck abelian category is cofibrantly generated, and we provide explicit generating (acyclic) cofibrations.
\end{abstract}

\subjclass{55P62, 55P91, 18N40, 18G35}
\keywords{Equivariant rational homotopy theory, model categories, chain complexes}

\maketitle

\section{Introduction}

The very heart of algebraic topology is about striving to capture features of topological spaces via tractable algebraic invariants. 
Singular cohomology, for example, is the cohomology of a cochain complex, which we treat as an algebraic model for the space, or rather its homotopy type. 
A natural question is: can one get ``perfect'' models, that is, models that  
are in one-to-one correspondence with homotopy types, 
and which are moreover computable?

This is hopelessly naive, but one can refine the question and restrict the spaces enough until one starts getting positive results. 
The first such result is due to Quillen \cite{Quillen-rht}. 
He considers simply-connected pointed spaces (or 1-reduced simplicial sets), 
taken up to rational homotopy equivalence.
On the algebraic side, he gives, most notably, differential graded Lie (dg Lie) and
differential graded cocommutative coalgebra (cdgc) models. 
To establish their equivalence means: to give these categories some model structure and to produce Quillen equivalences between them.
This is what Quillen achieved, but  the connection between spaces and these models was a bit roundabout, 
as it did not arise from direct Quillen equivalences but rather as zig-zags of Quillen equivalences.

A few years later, Bousfield and Gugenheim \cite{Bousfield--Gugenheim}, 
building on work by Sullivan \cite{FriedlanderGriffithsMorgan1972,Sullivan} 
gave commutative differential graded algebra (cdga) models for rational spaces which are nilpotent, which is more general than being simply-connected.
Unfortunately, the spaces and algebras need finite-type hypotheses for this to work. 
On the other hand, there is no need for base points, although they do give a pointed version for their equivalence as well. %
This equivalence of homotopy theories takes the form of a 
single Quillen adjunction whose induced adjunction on homotopy categories restricts to an equivalence on some full subcategories. 
We review this in \cref{sec:bg}, where we highlight that this is an instance of a \emph{partial Quillen equivalence} (more on this below).

Much later, new dg Lie models for the homotopy theory of rational spaces were produced. This was a collective effort a long time in the making. Some of the first papers in this direction were Getzler's and Hinich's \cite{Hinich, Getzler}. We shall exploit the approach using complete dg Lie algebras (cdgls) from \cite{LieModelsInTopology}, as it fits our framework most nicely. 
Another well-developed approach to modeling rational spaces with the closely related $L_\infty$-algebras can be found in \cite{NicoudVallette}, %
see also \cite{Roca} for an approach in terms of curved absolute $L_\infty$-algebras.

It is known that cdgls model connected nilpotent pointed rational Kan complexes \cite[Thm. 0.5]{LieCharacterizationBKcompletion}, 
and this approach presents some advantages to %
Quillen's dg Lie models.
Indeed, it is a direct comparison: it takes the form of a single partial Quillen equivalence. 
Moreover, it applies to nilpotent instead of simply-connected spaces. 
We review these results in \cref{sec:Lie}.

In this paper, we are only interested in the rational case, but
we mention in passing that arguably the strongest result thus far on algebraic models for spaces is due to Mandell, 
who produced $E_\infty$-algebra models for nilpotent finite-type spaces, not necessarily rational \cite{Mandell}. \medskip

Our aim is to explain how these algebraic models for rational spaces can be generalized to the $G$-equivariant setting in a clean, 
systematic way, where $G$ is a discrete group. 
We present a framework that could be used to give equivariant generalizations of other algebraic models for spaces, also not necessarily rational ones.

In order to achieve this we adopt Elmendorf's insight \cite{Elmendorf,MarcStephan} that one can model genuine $G$-equivariant homotopy theory via simplicial presheaves on the orbit category $\OG$. More precisely, the category of contravariant functors from $\OG$ to simplicial sets, endowed with the projective model structure, is Quillen equivalent to the genuine model structure on the category of $G$-simplicial sets (which is in turn Quillen equivalent to the analogous model category of $G$-topological spaces).

The following unifying definition will be very helpful: if $\C_0\subset \C$, $\D_0\subset \D$ are full subcategories and $\C\rightleftarrows \D$ is a Quillen adjunction, we say it is a \emph{partial Quillen equivalence on $\C_0$ and $\D_0$} if it satisfies the natural hypotheses that guarantee that the derived adjunction $\Ho(\C)\rightleftarrows \Ho(\D)$ between homotopy categories restricts
to an equivalence of categories on the full subcategories on the objects of $\C_0$ and $\D_0$, respectively.
Most importantly, the derived unit (resp. counit) on the %
objects of the corresponding subcategories must be weak equivalences. See \cref{sect-partial} for details. This simple definition serves to capture in a streamlined way the type of equivalence obtained by Bousfield--Gugenheim between rational Kan complexes and cdgas, and the type of equivalence obtained by Buijs--Félix--Murillo--Tanré between rational Kan complexes and cdgls.

In view of the above, to give cdga models or cdgl models for $G$-equivariant spaces, we can consider their categories of $\OG$-indexed presheaves with the projective model structure. 
Indeed, an adjunction $\C\rightleftarrows \D$ induces an adjunction $\C^I \rightleftarrows \D^I$ between functor categories, where $I$ is any 
small category,
and if the former is a Quillen equivalence of model categories, then the latter is so, too, with the projective model structures.
Since our examples are \emph{partial} Quillen equivalences, we first prove the following result. See \cref{lem-sciarappa} in the text for details on the technical conditions.

\begin{lemn}[\ref{lem-sciarappa}] Let $I$ be a small category. 
Let $\C\rightleftarrows \D$ be a partial Quillen equivalence on the subcategories $\C_0$, $\D_0$. Then assuming mild technical conditions automatically satisfied in our examples, the induced adjunction $\C^I_\proj \rightleftarrows \D^I_\proj$ is a partial Quillen equivalence on the subcategories $\C_0^I$, $\D_0^I$.
\end{lemn}

This lemma is not surprising, but it has the %
consequence of supplying the algebraic models that we desired. First, we can apply this to the partial
Quillen equivalence giving cdgl models for rational Kan complexes, namely $\mathcal L : \sSet \rightleftarrows \cDGL : \langle - \rangle$. We obtain a model category giving Lie algebraic models for $G$-equivariant spaces:

\begin{thmn}[\ref{cor-cdgl}] Let $G$ be a discrete group.
There is a Quillen adjunction 
\[
\mathcal{L}_*:\sSet^\OGop_\proj \rightleftarrows \cDGL^\OGop_\proj:\langle-\rangle_*.
\]
It is a partial Quillen equivalence on the full subcategories of $\OG$-indexed presheaves of %
connected nilpotent rational %
simplicial sets and the $\OG$-indexed presheaves of homologically nilpotent %
cdgls weakly equivalent to a connected one.
\end{thmn}

This result is theoretically satisfying but hard to use in practice until there is a theory of equivariant Lie minimal models that would provide computable cofibrant replacements. This will not be a problem with cdga models, as we will now explain.\medskip

We apply \cref{lem-sciarappa} to the partial Quillen equivalence $\APL: \sSet\rightleftarrows \CDGA^\op:\langle - \rangle$ of Bousfield--Gugenheim, to obtain the 
following result:

\begin{thmn}[\ref{cor-sciarappa-cdga}] Let $G$ be a discrete  group.
There is a  Quillen adjunction 
\[
\left(\APL\right)_*:\sSet^\OGop_\proj\rightleftarrows \left(\CDGA^\OG_\inj\right)^\op:\langle - \rangle_*.
\]
It is a partial Quillen equivalence on the $\OG$-indexed presheaves of connected,
nilpotent rational %
simplicial sets of finite type and the
$G$-systems of cohomologically connected %
cdgas weakly of finite type. Here $\inj$ indicates an injective model structure.
\end{thmn}

Note that due to the contravariance of Bousfield--Gugenheim's functors, on the cdga side the projective structure turns into the injective, and we do not get $\OG$-presheaves but rather \emph{covariant} functors from $\OG$, which we call \emph{$G$-systems of cdgas}.
There are different model structures on $G$-systems of cdgas with a similar relation to $G$-equivariant rational homotopy theory in the literature, see e.g. \cite{Golasinski98,Mandell_2002,Scull-2008}.
We explain how these are related to our work in \cref{sec-other-models}. We humbly believe our approach is more streamlined and systematic.

Next, we define a different model structure on $\CDGA^\OG$, for reasons that we make clear below. It is the right-transferred model structure via the forgetful functor $\CDGA^\OG \to (\Chz(\Vect))^\OG\cong \Chz(\Vect^\OG)$ to $G$-systems of cochain complexes of vector spaces, or equivalently, cochain complexes of $G$-systems of vector spaces, which we endow with the injective model structure. For this reason, we call it the \emph{$r(\inj)$ model structure} in $\CDGA^\OG$. To establish the existence of this model structure requires some work, done in \cref{sect-two-model}.

\begin{thmn}[\ref{thm-IJ-cdga}]The $r(\inj)$ model structure on $\CDGA^\OG$ exists. Its weak equivalences are the component-wise weak equivalences and its fibrations are the epimorphisms with degreewise injective kernel as an object of $\Vect^\OG$. It is cofibrantly generated. The fibrant objects are the systems of cdgas which are degreewise injective as objects of $\Vect^\OG$.
\end{thmn}

The proof of the above is an application of Kan's classical theorem on right-transferred model structures. In order for it to apply, we need the injective model structure on $\Chz(\Vect^\OG)$ to be cofibrantly generated. We prove this, giving explicit sets of generating (acyclic) cofibrations:

\begin{thmn}[\ref{thm-IJ}]  There is an \emph{injective} model structure on $\Chz(\Vect^\OG)$ whose weak equivalences are the quasi-isomorphisms, the cofibrations are the monomorphisms in positive degrees, and the fibrations are the epimorphisms with degreewise injective kernel. It is cofibrantly generated. Every object is cofibrant, and the fibrant objects are the complexes which are degreewise injective.
\end{thmn}

Note that this model structure, interpreted as the injective model structure on the category of functors from $\OG$ to non-negative cochain complexes $\Chz(\Vect)$, is known to be cofibrantly generated \cite[Prop. A.2.8.2]{LurieHTT}, but the sets of generating (acyclic) cofibrations available in a general injective model structure in a functor category are large and not too explicit; a similar phenomenon occurs in the injective model structure on unbounded chain complexes  \cite[Thm. 2.3.13]{Hovey-ModelCategories}, \cite[Prop. 3.13]{Beke}. %
Here, on the contrary, we give explicit generating (acyclic) cofibrations. They are defined using sphere and disk complexes %
over generators and over a \emph{Baer set}, i.e. a set of monomorphisms that detect injectivity, just like inclusions of ideals into a commutative ring detect injectivity in the category of modules over it. We refer to (\ref{subsub-IJ}) for more details. 

We then explain in \cref{thm-grothendieck} how to generalize this to non-negative cochain complexes in more general abelian categories $\A$, a result of independent interest:

\begin{thmn}[\ref{thm-grothendieck}] 
Let $\A$ be a Grothendieck abelian category with a set of generators $\{P_H\}_{H\in \O}$ indexed by a set $\O$. Let $\{U_t\to V_t\}_{t\in T}$ be a Baer set in $\A$. The \emph{injective} model structure on $\Chz(\A)$ exists: its weak equivalences are the quasi-isomorphisms, the fibrations are the epimorphisms with degreewise injective kernel, and the cofibrations are the monomorphisms in positive degrees. It is cofibrantly generated, with generating acyclic cofibrations given by %
\[J=\left\{0\to D^n(P_H)\right\}_{\substack{n\geq 1 \\ H\in \O}} \quad \cup \quad \left\{D^n(U_t)\to D^n(V_t)\right\}_{\substack{n\geq 1 \\ t\in T}}\]
        and generating cofibrations given by 
\[I=\left\{S^n(P_H) \to D^n(P_H)\right\}_{\substack{n\geq 1 \\ H \in \O}} \quad \cup \quad \left\{S^0(P_H) \to 0 \right\}_{H\in \O} \quad \cup \quad \left\{S^n(U_t) \to D^n(V_t)\right\}_{\substack{n\geq 1 \\ t\in T}}.\]
Every object is cofibrant, and the fibrant objects are the complexes which are degreewise injective.
\end{thmn}

Going back to the $r(\inj)$ model structure on $\CDGA^\OG$: while this is not the same model structure as the injective model structure, they are Quillen equivalent, as we prove in \cref{prop-comparison-inj-r(inj)}. From this, we obtain that the $r(\inj)$ model structure also serves to model $G$-equivariant spaces:

\begin{thmn}[\ref{thm-spaces-r(inj)}] 
There is a  Quillen adjunction \[\left(\APL\right)_*:\sSet^\OGop_\proj\rightleftarrows \left(\CDGA^\OG_{r(\inj)}\right)^\op:\langle - \rangle_*.\] 
It is a partial Quillen equivalence on the $\OG$-indexed presheaves of connected,
nilpotent rational %
simplicial sets of finite type and the $G$-systems of cohomologically connected %
cdgas weakly of finite type.
\end{thmn}

The reason this is a more interesting theorem than the analogous \cref{cor-sciarappa-cdga} from before, which used the injective model structure on $\CDGA^\OG$, is that there is a theory of minimal models giving cofibrant replacements in $\CDGA^\OG_{r(\inj)}$, thus rendering these algebraic models computable. This theory was worked out by \cite{Triantafillou} and it needs the added hypothesis that $G$ is finite, see also \cite{Scull-2002}. In \cite[Cor. 4.3]{Scull-2008} it was proven that minimal models are cofibrant approximations in the model category we call $(\CDGA^1)^\OG_{r(\inj)}$: the superscript indicates we are restricting to \emph{connected} cdgas. We work out the details of this model structure in \cref{sect-rinj-connected}. This corrects several problems in \cite{Scull-2008}, as we explain in \cref{sec-scull}. The proof given there for the existence of the $r(\inj)$ model structure on $G$-systems of connected cdgas was incomplete in some parts and wrong in others, notably in the description of the sets of generating (acyclic) cofibrations. Our correction also has the advantage of making all these structures and their interactions clearer. We repeatedly use the classical theorem of Kan which establishes the existence of transferred model structures along a right adjoint.

We also note that in the upcoming article \cite{Golub}, Golub is working towards understanding minimal models of cdgas in other categories of diagrams, replacing $\OG$ with a more general small category. \medskip

We now make some final comments. There is a growing interest in the theory of systems of cdgas. 
For example, Sati--Schreiber maintain a research program on the connection between
equivariant rational homotopy theory applied to M-theory and higher gauge theory, 
see for instance \cite{SatiSchreiber2025CharacterMap}, where these systems are an essential tool.
On the other hand, Santhanam--Thandar \cite{SanthanamThandar} work toward the understanding of the notion of injectivity (cofibrancy) and formality of these systems.

\subsection{Notations and conventions} All vector spaces are taken to be over $\Q$. A graded vector space or (co)chain complex is said to be of \emph{finite type} if it is finite-dimensional in each degree. 
Denote by 
\[
\Ch^*, \qquad \Ch_*, \qquad \Chz,\qquad  \Chaz
\]
the categories of unbounded cochain complexes, unbounded chain complexes,
non-negative cochain complexes and non-negative chain complexes respectively, 
all of them of rational vector spaces.  We consider these categories endowed with the standard symmetric monoidal structure whose symmetry has a Koszul sign.

If $\C$ is a category and $I$ is a small category, we denote by $\Fun(I,\C)$ or by $\C^I$ the category of functors from $I$ to $\C$. 
When we write an adjunction as $\C\rightleftarrows \D$, the left adjoint is always on top.

We do not assume model categories to have functorial (co)fibrant replacements. We denote by $\Ho(\C)$ the homotopy category of a model category, which has the same objects as $\C$. 
If $I$ is a class of maps of $\C$, 
denote by $I^{\boxslash}$ the class of maps which have the right lifting property with respect to the maps of $I$, 
and by ${}^{\boxslash}I$ the class of maps which have the left lifting property with respect to the maps of $I$.

If $\C$ is a model category, then we consider $\C^\op$ as a model category with the opposite model structure \cite[Rem. 1.1.7]{Hovey-ModelCategories}.

Throughout the article, $G$ denotes a discrete group.

\subsection{Acknowledgments}
The authors would like to thank James Gillespie, Nima Rasekh and Karol Szumiło for helpful conversations, and Katsuhiko Kuribayashi for his useful comments on the first version of this work. 
The first author has been partially supported by the MICINN grant PID2023-149804NB-I00,
the Junta de Andalucía grant ProyExcel-00827, 
and the Junta de Andalucía research group grupoPaidi-G-FEDER-FQM264.

\section{Generalities}

We start by gathering some elementary results that will be used in the rest of the paper.

\subsection{Transferred model structures}

\begin{definition}
    Let $L:\C\rightleftarrows \D:R$ be an adjunction between model categories.
    We say the model structure on $\D$ is \emph{right-transferred via $R$} if 
    the weak equivalences (resp. fibrations) of $\D$ are the maps $f$ such that $R(f)$ is a weak equivalence (resp. fibration) in $\C$.
\end{definition}

\begin{remarks} \label{rmk-cofgen}
\begin{enumerate}
\item In the situation above, if the model structure on $\D$ is right-transferred via $R$,
then the adjunction $L\dashv R$ is a Quillen adjunction, since $R$ preserves (acyclic) fibrations.
    \item
If $L:\C\rightleftarrows \D:R$ is an adjunction where $\C$ is a cofibrantly generated model category with set of generating cofibrations $I$ and set of generating acyclic cofibrations $J$,
there are different conditions that guarantee that the right-transferred model structure on $\D$ exists and has sets of generating (acyclic) cofibrations given by $L(I), L(J)$ respectively, 
see e.g. \cite[Thm. 11.3.2]{Hirschhorn}, \cite[Thm. 4.3.3]{Fresse-Teichmuller2} for one, or \cite[Rmk. 1.23]{LieModelsInTopology}, 
or \cite[Lem. 2.2]{drummond-cole--hackney}. %
\item \label{item-balchin}If $L:\C\rightleftarrows \D:R$ is an adjunction where $\C$ is
a cofibrantly generated model category and the right-transferred model structure on $\D$ via $R$ exists, 
then as soon as $L$ preserves small objects, this model structure on $\D$ is cofibrantly generated with sets of generating (acyclic) cofibrations $L(I), L(J)$ \cite[Prop. 4.4.4]{balchin}. %
For $L$ to preserve small objects,
it suffices for example that $\D$ is locally presentable (for in that case every object in $\D$ is small), 
or that $R$ preserves filtered colimits. In particular, if $L\dashv R$ is
the free-forgetful adjunction for algebras over some operad, then this hypothesis is automatic \cite[Prop. 3.3.1]{Fresse-2009}. 
\end{enumerate}
\end{remarks}

\begin{proposition} \label{prop-overunder} \cite[Rmk. 3.10, 3.11]{dwyer-spalinski} \cite[Thm. 1.20]{Hirschhorn2021}. %
Let $\C$ be a model category and $A\in \C$. Then the overcategory $\C_{/A}$ (resp. the undercategory $\C_{A/}$)
has a model structure where a map $f$ over $A$ (resp. under $A$) is a weak equivalence or (co)fibration 
if and only if it is a weak equivalence or (co)fibration in $\C$. If $\C$ is cofibrantly generated, 
then $\C_{/A}$ and $\C_{A/}$ are cofibrantly generated, where the generating (acyclic) 
cofibrations are the generating (acyclic) cofibrations of $\C$ which are over or under $A$ in all possible ways. %
\end{proposition}

\subsection{Some classical model structures}

We now recall some classical model structures that will be used in the sequel.

\begin{theorem}
    There is a %
    model structure on the category $\sSet$ of simplicial sets
    whose weak equivalences are the weak homotopy equivalences, 
    the cofibrations are the monomorphisms (i.e. levelwise injections), and the fibrations are the Kan fibrations.
    It is cofibrantly generated by the boundary and horn inclusions. The fibrant objects are the Kan complexes and every object is cofibrant.
\end{theorem}

\begin{example}
    Applying \cref{prop-overunder} to $*\in \sSet$ gives a model structure on the undercategory $\sSet_{*/}$, which by definition is the category of pointed simplicial sets $\sSet_*$.
\end{example}

We will also need the standard model structures on non-negative (co)chain complexes. 
Since we are over a field, there is no distinction of projective vs. injective model structure. 

First, we define some objects. For $n\in \Z$ define the \emph{sphere} $S^n\in \Ch^*$ to be 
the cochain complex which has $\Q$ in degree $n$ and is zero elsewhere.
Define the \emph{disk} $D^n\in \Ch^*$ to be the cochain complex which has $\Q$
in degrees $n-1$ and $n$ with the identity between them as differential, 
and is zero elsewhere. In chain complexes we make the same definitions, only the differential in the disk goes down instead of up.

There are obvious inclusion maps $S^n\to D^n$ in cochain complexes and maps $S^{n-1}\to D^n$ in chain complexes. %
We will also consider the inclusions of the zero (co)chain complex into disks.

\begin{theorem}  
\label{prop-model-chz} \cite[4.4]{Bousfield--CosimplicialResolutions},
\cite[1.6]{Goerss--Schemmerhorn}, %
\cite[Section 5.1]{Fresse-Teichmuller2}%
There is a model structure on $\Chz$ whose weak equivalences are the quasi-isomorphisms, 
the cofibrations are the degreewise injections in positive degrees, and the fibrations are the degreewise surjections.
It is cofibrantly generated, with generating cofibrations given by the inclusions $S^n\to D^n$ in degrees $n\geq 1$ as well as the map $S^0\to 0$, 
and with generating acyclic cofibrations given by the inclusions $0\to D^n$ for $n\geq 1$.
\end{theorem}

\begin{theorem}  
\label{prop-model-chaz}
\cite[II.4.(5)]{Quillen-HomotopicalAlgebra}, %
\cite[7.2, 7.19]{dwyer-spalinski}, 
\cite[1.5, 3.4]{Goerss--Schemmerhorn} %
There is a model structure on $\Chaz$ whose weak equivalences are the quasi-isomorphisms, 
the cofibrations are the degreewise injections, and the fibrations are the degreewise surjections in positive degrees.
It is cofibrantly generated, with generating cofibrations given by the inclusions $S^{n-1}\to D^n$ in degrees $n\geq 1$ as well as the map $0\to S^0$, 
and with generating acyclic cofibrations given by the inclusions $0\to D^n$ for $n\geq 1$.
\end{theorem}

\subsection{Some rational homotopy theory}

We now recall some terminology from rational homotopy theory. Let $X$ be a connected pointed %
Kan complex. 

\begin{itemize}
    \item We say $X$ is \emph{nilpotent} if $\pi_1(X)$ is a nilpotent group 
    and $\pi_n(X)$ is a nilpotent $\pi_1(X)$-module for all $n\geq 2$. 
\item A nilpotent $X$ is \emph{rational} if the integral homology groups in degrees $\geq 1$ 
are all uniquely divisible (equivalently, $\Q$-vector spaces).
\item A nilpotent $X$ is \emph{of finite type} if its rational homology vector spaces  are finite-dimensional.
\end{itemize}
If $X$ is not pointed, we say it is \emph{nilpotent} or \emph{rational} or \emph{of finite type} 
if it is so for some (hence any) choice of base point. 

The above terminology extends to arbitrary simplicial sets by taking weakly equivalent Kan complexes. Therefore, each of these three classes of simplicial sets is closed under weak equivalences.

\subsection{Some equivariant homotopy theory}

We now recall some elementary results from equivariant homotopy theory. The \emph{orbit category} $\OG$ of the discrete group $G$ has as objects the subgroups $K\leq G$, 
and the morphisms from $K$ to $K'$ are the set-theoretic $G$-equivariant maps $G/K\to G/K'$. We now recall Elmendorf's theorem, which explains the appearance of $\OG$-indexed presheaves in the world of genuine equivariant homotopy theory.

Let $\GsSet$ be the category of $G$-objects in $\sSet$, i.e. $\Fun(BG, \sSet)$ where $BG$ is the category with one object associated to the group $G$. Defining a weak equivalence (resp. fibration) to be a map which is a weak equivalence (resp. fibration) in $\sSet$ under every $H$-fixed point functor, $H\leq G$ defines the \emph{genuine model structure} on $\GsSet$ \cite{MarcStephan}. This is what is commonly known as the genuine equivariant homotopy theory of $G$-Kan complexes. %

\begin{theorem}[Elmendorf] \label{thm-elmendorf} \cite{Elmendorf}, \cite[Thm. 2.10, Ex. 2.14]{MarcStephan} %
There is a Quillen equivalence $\sSet^\OGop \rightleftarrows \GsSet$, where the functor category $\sSet^\OGop$ has the projective model structure. The left Quillen adjoint is evaluation at the trivial subgroup, and the right Quillen adjoint takes a $G$-space to its $\OG$-indexed presheaf of fixed points.
\end{theorem}

Thus, we can consider $\sSet^\OGop$ with the projective model structure as giving us the genuine $G$-equivariant homotopy theory of spaces. %
Under the above Quillen equivalence, %
the presheaves of nilpotent rational Kan complexes of finite type correspond to the $G$-simplicial sets all of whose fixed point spaces are nilpotent, rational and of finite type. %

\section{Partial Quillen equivalences} \label{sect-partial}

We now introduce the key notion of ``partial Quillen equivalence'', which is similar to that in \cite[Page 12]{Sciarappa2017} but more general. %

\subsection{Definition and basic properties}

We first recall the notion of derived (co)unit. Let $L:\C\rightleftarrows \D:R$ be a Quillen adjunction between model categories. Let $q_X:QX\to X$ be a cofibrant replacement of the object $X\in \C$, i.e. $q_X$ is an acyclic fibration and $QX$ is cofibrant; let $p_Y:Y\to PY$ be a fibrant replacement of the object $Y\in \D$, i.e. $p_Y$ is an acyclic cofibration and $PY$ is fibrant. (We will henceforth use these notations for (co)fibrant replacements without introducing them each time.) %
Then a \emph{derived unit} of $L\dashv R$ at the object $X$ is a map of the form
\[\xymatrix{X\ar[r]^-{\eta_X} & RLX \ar[r]^-{Rp_{LX}} & RPLX,}\]
and a \emph{derived counit} of $L\dashv R$ at the object $Y$ is a map of the form
\[\xymatrix{LQRY \ar[r]^-{Lq_{RY}} & LRY \ar[r]^-{\epsilon_Y} & Y.
} \]

We will use the following construction several times. If $f:A\to B$ is a map in $\C$, then we can construct a map $Pf:PA\to PB$ such that $Pf\circ p_A=p_B\circ f$, using the lifting properties of $\C$:
\begin{equation}\label{lifts-of-maps}\tag{$\clubsuit$}\begin{tikzcd}
	A & PB \\
	PA & {*}
	\arrow["{p_B\circ f}", from=1-1, to=1-2]
	\arrow["{p_A}"', tail, from=1-1, to=2-1]
	\arrow[two heads, from=1-2, to=2-2]
	\arrow["Pf"', dotted, from=2-1, to=1-2]
	\arrow[from=2-1, to=2-2]
\end{tikzcd}\end{equation}

\begin{remark} \label{rmk-derived-unit}
The definition of a derived (co)unit depends on the choice of (co)fibrant replacement only in an inessential way: If $p'_{LX}: LX\to P'LX$ is another fibrant replacement, consider the following lifting problem.
\[\begin{tikzcd}
	LX & {P'LX} \\
	PLX & {*}
	\arrow["{p'_{LX}}", "\sim"', from=1-1, to=1-2]
	\arrow["{p_{LX}}"',"\sim", tail, from=1-1, to=2-1]
	\arrow[two heads, from=1-2, to=2-2]
	\arrow["w"', dotted, from=2-1, to=1-2]
	\arrow[from=2-1, to=2-2]
\end{tikzcd}\]
We get a weak equivalence $w:PLX\to P'LX$ between fibrant objects, so $Rw$ is a weak equivalence; hence one derived unit at $X$ is a weak equivalence if and only if the other is. Dually for the derived counit. All conditions below imposed on derived (co)units are therefore independent of these choices. Thus, we will speak of ``the'' derived (co)unit.
\end{remark}

\begin{definition} \label{def-partial}
    Let $L:\C\rightleftarrows \D:R$ be a Quillen adjunction.
    Suppose $\C_0\subset \C$ and $\D_0\subset \D$ are full subcategories such that the following hold:
    \begin{enumerate}
    \item Every $X\in \C_0$ admits a cofibrant replacement $q_X:QX\to X$ with
$QX\in \C_0$, and every $Y\in \D_0$ admits a fibrant replacement
$p_Y:Y\to PY$ with $PY\in \D_0$.
           \item If $X\in \C_0$ is cofibrant, then $LX\in \D_0$, and if $Y\in \D_0$ is fibrant, then $RY\in \C_0$.  %
        \item If $X\in \C_0$ is cofibrant, then the derived unit $X\to RPLX$ is a weak equivalence.
        Dually, if $Y\in \D_0$ is fibrant, then the derived counit $LQRY\to Y$ is a weak equivalence. %
    \end{enumerate}
    Then we say $L:\C\rightleftarrows \D:R$ is a \emph{partial Quillen equivalence on the subcategories $\C_0$ and $\D_0$}.
\end{definition}

\begin{remark}
    For $\C_0=\C$ and $\D_0=\D$, we recover the notion of a Quillen equivalence.
\end{remark}

We can summarize the idea behind this definition as follows: An adjunction is an equivalence if and only if its unit and counit maps are isomorphisms; a Quillen adjunction is a Quillen equivalence if and only if its derived unit (resp. counit) maps %
are equivalences; and generalizing that to model categories endowed with subcategories which interact well with the Quillen adjunction gives the definition of a partial Quillen equivalence. 

\begin{remark} \label{rmk-condition-one}
Condition (1) on the subcategory $\C_0$ holds as soon as $\C_0$ is closed under weak equivalences, in the sense that if $X\in \C$ is weakly equivalent %
        to $X_0$ with $X_0\in \C_0$, then $X\in \C_0$ too. It also holds if every object of $\C$ is cofibrant. Dually, Condition (1) on the subcategory $\D_0$ holds as soon as $\D_0$ is closed under weak equivalences or if every object in $\D$ is fibrant.
\end{remark}

We now give a lemma which proves that if one has a partial Quillen equivalence where every object in $\C$ is cofibrant and every object in $\D$ is fibrant, then enlarging the subcategories as little as possible so they are closed under weak equivalences also gives a partial Quillen equivalence. In other words: we now consider the closure of $\C_0\subset \C$ under weak equivalences. We denote it $\widetilde{\C_0}$, and we can describe it as the full subcategory of $\C$ on the objects which are weakly equivalent to objects in $\C_0$. We define $\widetilde{\D_0}$ analogously.

\begin{lemma} \label{lem-saturation} Let $L:\C\rightleftarrows \D:R$ be a partial Quillen equivalence on the subcategories $\C_0\subset \C$, $\D_0\subset \D$. Assume every object in $\C$ is cofibrant and every object in $\D$ is fibrant. Then $L\dashv R$ is also a partial Quillen equivalence on the subcategories $\widetilde{\C_0}$ and $\widetilde{\D_0}$.
\begin{proof} Condition (1) is automatic since every object in $\C$ is cofibrant and every object in $\D$ is fibrant.

Let us prove Condition (2). Let $X\in \widetilde{\C_0}$. Then $X\simeq X_0$ with $X_0\in \C_0$ via a zig-zag of weak equivalences in $\C$. Applying $L$ to it gives a zig-zag of weak equivalences in $\D$ since $L$ preserves weak equivalences between cofibrant objects, so $LX\simeq LX_0\in \D_0$, proving $LX\in \widetilde{\D_0}$. The condition on $R$ is proven dually.

We now prove Condition (3). Let $X\in \widetilde{\C_0}$. Then $X\simeq X_0$ with $X_0\in \C_0$ via a zig-zag of weak equivalences in $\C$. Consider the derived units of all the objects in the zig-zag. We connect them via a diagram: we draw the beginning of it corresponding to the first step $w_1$ in the zig-zag.
\[\begin{tikzcd}
	X & RLX & RPLX \\
	{X_1} & {RLX_1} & {RPLX_1}
	\arrow["{\eta_X}", from=1-1, to=1-2]
	\arrow["{Rp_{LX}}", from=1-2, to=1-3]
	\arrow["{w_1}", "\sim"', from=2-1, to=1-1]
	\arrow["{\eta_{X_1}}"', from=2-1, to=2-2]
	\arrow["{RLw_1}"', from=2-2, to=1-2]
	\arrow["{Rp_{LX_1}}"', from=2-2, to=2-3]
	\arrow["{RPLw_1}"', "\sim", from=2-3, to=1-3]
\end{tikzcd}\]
Here the map $PLw_1$ is defined as in (\ref{lifts-of-maps}). In the ladder diagram above, if one of the bottom or top compositions is a weak equivalence, then so is the other, by 2-out-of-3. Extending the ladder downwards along the zig-zag, we reach $X_0$ whose derived unit is a weak equivalence. The condition of being a weak equivalence extends upwards and thus the derived unit of $X$ is also a weak equivalence. The condition on $R$ is proven dually.
\end{proof}
\end{lemma}

We now prove the most important feature of a partial Quillen equivalence, namely: the induced derived adjunction $\Ho(\C)\rightleftarrows \Ho(\D)$ 
restricts to an equivalence of categories between the full subcategories of $\Ho(\C)$ and 
$\Ho(\D)$ on the objects of $\C_0$ and $\D_0$, respectively. This is an adaptation of a well-known characterization of Quillen equivalences, but we give a full proof because we were unable to find a complete reference for that fact other than \cite[Prop. 1.3.13]{Hovey-ModelCategories}, which assumes that model categories come endowed with functorial factorizations. %

\begin{proposition} Let $L:\C\rightleftarrows \D:R$ be a Quillen adjunction. Let $\C_0\subset \C$ and $\D_0\subset \D$ be full subcategories. Denote by $h\C_0$ (resp. $h\D_0$) the full subcategory of $\Ho(\C)$ (resp. $\Ho(\D)$) on the objects of $\C_0$ (resp. $\D_0$).

Assume that (1) and (2) in \cref{def-partial} hold. Then the derived adjunction $\Ho(\C)\rightleftarrows \Ho(\D)$ restricts to an adjunction $h\C_0\rightleftarrows h\D_0$, and the following are equivalent:
\begin{enumerate}
    \item[a)] Condition (3) in \cref{def-partial} holds, so $L\dashv R$ is a partial Quillen equivalence on $\C_0$ and $\D_0$. %
    \item[b)] A map $LX\to Y$ is a weak equivalence if and only if its transpose $X\to RY$ is a weak equivalence, for all $X\in \C_0$ cofibrant and $Y\in \D_0$ fibrant.
        \item[c)] The adjunction $h\C_0\rightleftarrows h\D_0$ is an adjoint equivalence.
\end{enumerate}
\end{proposition}
\begin{proof} We can compute the left derived functor $\Ho(\C)\to \Ho(\D)$ on $X\in \C$ as $L(QX)$ where $QX$ is a cofibrant replacement of $X$ inside $\C_0$, by (1). If $X\in \C_0$, then by (2) we have that $L(QX)\in \D_0$. Dually for $R$, so %
the derived adjunction restricts to the desired subcategories.

We now prove (a) implies (b). Let $f:LX\to Y$ be a weak equivalence, where $X\in \C_0$ is cofibrant and $Y\in \D_0$ is fibrant. Let $p_Y:Y\to PY$ be a fibrant replacement. %
By (\ref{lifts-of-maps}), we have a map $Pf:PLX\to PY$ such that $Pf\circ p_{LX}= p_Y\circ f$. By 2-out-of-3, it is a weak equivalence. We have the following commutative diagram in $\C$, where the top curved arrow is the derived unit at $X$, which is a weak equivalence by assumption:
\[\begin{tikzcd}
	X & RLX & RPLX \\
	& RY & RPY
	\arrow["{\eta_X}", from=1-1, to=1-2]
	\arrow["\sim", curve={height=-22pt}, from=1-1, to=1-3]
	\arrow["{Rp_{LX}}", from=1-2, to=1-3]
	\arrow["Rf"', from=1-2, to=2-2]
	\arrow["RPf","\sim"', from=1-3, to=2-3]
	\arrow["{Rp_Y}"', "\sim", from=2-2, to=2-3]
\end{tikzcd}\]
The right vertical map and the bottom horizontal map are weak equivalences because $R$ takes weak equivalences between fibrant objects to weak equivalences. Now, the composition $Rf\circ \eta_X:X\to RY$ is the transpose of $f$: it is a weak equivalence by 2-out-of-3.

Dually, if we are given a weak equivalence $X\to RY$ with $X\in \C_0$ cofibrant and $Y\in \D_0$ fibrant, then its transpose will be a weak equivalence thanks to the derived counit at $Y$ being a weak equivalence. \medskip

To see (b) implies (a), let $X\in \C_0$ be a cofibrant object. Since $LX\in \D_0$, we have a fibrant replacement $p_{LX}:LX\to PLX$ such that $PLX\in \D_0$. Therefore, by hypothesis, the transpose of $p_{LX}$ is a weak equivalence. But this map is exactly the derived unit of $X$. The assertion for the derived counit is proven dually. \medskip

We now prove that (a) is equivalent with (c). The adjunction $h\C_0\rightleftarrows h\D_0$ is an adjoint equivalence if and only if its unit and counit are isomorphisms. The unit at $X\in \C_0$ is the composition in $h\C_0$:
\[\xymatrix{X \ar[r]^-{q_X^{-1}}_-\cong & QX \ar[r]^-{\eta_{QX}} & RLQX \ar[r]^-{Rp_{LQX}} & RPLQX.}\]
Here $QX\in \C_0$ and the fibrant replacement $PLQX$ of $LQX\in \D_0$ are chosen
using Condition (1), so that $RPLQX\in \C_0$ by (2) and the displayed composite
is a morphism of $h\C_0$; the intermediate objects need not lie in $\C_0$, which
is harmless since $h\C_0\subset \Ho(\C)$ is full.

This map is an isomorphism in $h\C_0$ if and only if the composition $Rp_{LQX}\circ \eta_{QX}$ is a weak equivalence. We will now check that for this to be true it is necessary and sufficient for the derived unit at $X$ to be a weak equivalence for all cofibrant $X$. Sufficiency is obvious, %
so we now check necessity.

By (\ref{lifts-of-maps}), %
there exists a map $PLq_X:PLQX\to PLX$ such that $PLq_X\circ p_{LQX}=p_{LX} \circ Lq_X$. Applying $R$ to this gives the right square in the following commutative diagram: %
\[\begin{tikzcd}
	QX & RLQX & RPLQX \\
	X & RLX & RPLX
	\arrow["{\eta_{QX}}", from=1-1, to=1-2]
	\arrow["{q_X}","\sim"', from=1-1, to=2-1]
	\arrow["{Rp_{LQX}}", from=1-2, to=1-3]
	\arrow["{RLq_X}", from=1-2, to=2-2]
	\arrow["{RPLq_X}", "\sim"', from=1-3, to=2-3]
	\arrow["{\eta_X}"', from=2-1, to=2-2]
	\arrow["{Rp_{LX}}"', from=2-2, to=2-3]
\end{tikzcd}\]
The %
right vertical arrow is a weak equivalence because left (resp. right) Quillen functors send weak equivalences between cofibrant (resp. fibrant) objects to weak equivalences. The bottom composition is the derived unit at $X$. By 2-out-of-3, if the top row is a weak equivalence then the bottom row is a weak equivalence, which is what we wanted to prove. The counit is treated dually.
\end{proof}

We now observe that composition of partial Quillen equivalences is a partial Quillen equivalence. The proof is simple so we omit it.

\begin{lemma} \label{lem-composition-partial}
    Let %
$\begin{tikzcd}
	\C & \D & \E
	\arrow[""{name=0, anchor=center, inner sep=0}, "{L_1}", from=1-1, to=1-2]
	\arrow[""{name=1, anchor=center, inner sep=0}, "{R_1}", shift left=3, from=1-2, to=1-1]
	\arrow[""{name=2, anchor=center, inner sep=0}, "{L_2}", from=1-2, to=1-3]
	\arrow[""{name=3, anchor=center, inner sep=0}, "{R_2}", shift left=3, from=1-3, to=1-2]
	\arrow["\dashv"{anchor=center, rotate=-90}, draw=none, from=0, to=1]
	\arrow["\dashv"{anchor=center, rotate=-90}, draw=none, from=2, to=3]
\end{tikzcd}$ be two Quillen adjunctions. Suppose $\C_0\subset \C$, $\D_0\subset \D$, $\E_0\subset \E$ are full subcategories and $L_1\dashv R_1$, $L_2\dashv R_2$ are partial Quillen equivalences on them. Then the composed adjunction $L_2\circ L_1 \dashv R_1\circ R_2$ is a partial Quillen equivalence on $\C_0$ and $\E_0$.
\end{lemma}

\subsection{Partial Quillen equivalences on functor categories}

We will now explain how a partial Quillen equivalence induces a partial Quillen equivalence on functor categories. Recall that if $\C$ is a model category and $I$ is a small category, then the \emph{projective model structure} on the functor category $\C^I$ has component-wise fibrations and weak equivalences, whereas the \emph{injective model structure} has component-wise cofibrations and weak equivalences. We will use the following existence theorems: The projective model structure on $\C^I$ exists provided $\C$ is cofibrantly generated \cite[Thm. 11.6.1]{Hirschhorn}, and the injective model structure on $\C^I$ exists provided $\C$ is \emph{combinatorial} \cite[Sect. A.2.6]{LurieHTT}, which means $\C$ is cofibrantly generated and locally presentable.

If $F:\C\to \D$ is a functor and $I$ is a small category, we let $F_*:\C^I\to \D^I$ denote the functor given by post-composition with $F$.

\begin{lemma}
\label{lem-adjunction-on-fun}
Let $L:\C\rightleftarrows \D:R$ be a Quillen adjunction (resp. Quillen equivalence) of model categories. 
Let $I$ be a small category. 
Then there is an induced adjunction $L_*:\C^I\rightleftarrows \D^I:R_*$ by post-composition, and 
it is a Quillen adjunction (resp. Quillen equivalence) between the projective or injective model structures, provided they exist. 
\end{lemma}

\begin{proof} This is straightforward, see  e.g. in \cite[A.2.8.6]{LurieHTT} or \cite[11.6.1, 11.6.5]{Hirschhorn} for the projective case.
\end{proof}

We will now prove that this induced Quillen adjunction is a partial Quillen equivalence if the original adjunction is a partial Quillen equivalence, as observed in \cite[Lem. 5.1]{Sciarappa2017}, although our proof is a bit different and our assumptions more general. %
This is not a surprising result, %
but we have not found any published proof for it, so we prove it in detail. It is an important step in the categorical adaptation of algebraic models to the equivariant setting.

\begin{lemma}
\label{lem-sciarappa} 
Let $I$ be a small category. 
Let $L: \C\rightleftarrows \D: R$ be a partial Quillen equivalence on the subcategories $\C_0$, $\D_0$. 
Assume moreover that $\C_0^I$ and $\D_0^I$ satisfy Condition (1) of \cref{def-partial}. 
Assume the projective model structures $\C^I_\proj$, $\D^I_\proj$ exist. Then the induced adjunction $L_*: \C^I_\proj \rightleftarrows \D^I_\proj:R_*$ is a partial Quillen equivalence on the subcategories $\C_0^I$, $\D_0^I$.

\begin{proof}
The first condition holds by hypothesis. %
The second condition follows from cofibrations in $\C^I$ being in particular levelwise cofibrations \cite[Prop. 11.6.3]{Hirschhorn}, which implies that cofibrant objects in $\C^I$ are levelwise cofibrant. The corresponding statements for fibrations are also true and even characterize them, by construction.

We need to check the third condition, which states: if $Y\in \D_0^I$ is a fibrant object, then the derived counit is a weak equivalence, and dually the derived unit of a cofibrant object in $\C_0^I$ is a weak equivalence. We give the proof of the first statement, the second one is proven dually.

Let $Y\in \D_0^I$ be a fibrant object, i.e.  $Y_i$ is fibrant for all $i\in I$. %
Let $\underline{q}_X:\underline{Q}X\to X$ denote cofibrant replacements for $X\in \C^I$. We want to prove that the composition
    \[\xymatrix{L_*\underline{Q}(R_*Y) \ar[r]^-{L_* \underline{q}_{R_*Y}} & L_*R_*Y \ar[r]^-{(\epsilon_*)_Y} & Y}\]
    is a weak equivalence in $\D^I$. Here $\epsilon_*$ denotes the counit of $L_*\dashv R_*$, which is merely component-wise the counit $\epsilon$ of $L\dashv R$. Therefore, we need to check  that
        \[\xymatrix{L(\underline{Q}(R_*Y)_i) \ar[r]^-{L (\underline{q}_{R_*Y})_i} & LRY_i \ar[r]^-{\epsilon_{Y_i}} & Y_i } \]
        is a weak equivalence in $\D$ for all $i\in I$.

But now note that we can take this map as the derived counit of $L\dashv R$ at $Y_i$, which is a weak equivalence by assumption. Indeed, $(\underline{q}_{R_*Y})_i:\underline{Q}(R_*Y)_i \to RY_i$ is an acyclic fibration and $\underline{Q}(R_*Y)_i$ is cofibrant in $\C$, since $\underline{Q}(R_*Y)$ is cofibrant. %
\end{proof}
\end{lemma}

\begin{remark} \label{rmk-homotopically-ff}In the proof of \cref{lem-sciarappa}, to prove that the derived unit of $L_*\dashv R_*$ on a cofibrant object is a weak equivalence, we use that the same holds true for $L\dashv R$: we do not need to use that the derived counit of $L\dashv R$ is a weak equivalence on fibrant objects for this. This proves an additional statement:

First, say $L$ is a \emph{partially homotopically fully faithful functor on the subcategories $\C_0$ and $\D_0$} if conditions (1) and (2) from \cref{def-partial} are satisfied, as well as the derived unit part of (3). Then we have proven the following: if $L$ is a partially homotopically fully faithful functor on $\C_0$ and $\D_0$ and $\C_0^I, \D_0^I$ satisfy Condition (1), then $L_* $ is partially homotopically fully faithful on $\C_0^I$ and $\D_0^I$.
\end{remark}

\begin{remark} \label{rmk-inheritance}
Condition (1) for $\C_0^I$ does not follow from Condition (1) for
$\C_0$ in general: a cofibrant replacement of a functor in $\C_0^I$ need not be in $\C_0^I$. It does follow %
if $\C_0$ is closed under weak equivalences (hence so is $\C_0^I$). On the other hand, if every object in $\C$ is cofibrant, this does not imply that every object in $\C^I$ is cofibrant in general, %
and Condition (1) for $\C_0^I$ in that case needs to be assumed.  Happily, Condition (1) for $\D_0^I$ follows both from $\D_0$ being closed under weak equivalences or every object of $\D$ being fibrant, since then every object in $\D^I$ is fibrant.

In conclusion, if the conditions from \cref{rmk-condition-one} hold for $\C_0$ or $\D_0$, then they do for $\C_0^I$ and $\D_0^I$ except possibly in the situation where all the objects of $\C$ are cofibrant. In that case, we can still say something, as we will see in the following lemma.
\end{remark}

In the following variation of \cref{lem-sciarappa}, we use the notations $\widetilde{\C_0}, \widetilde{\D_0}$ as in \cref{lem-saturation}.

\begin{lemma} \label{lem-sciarappa2} Let $L:\C\rightleftarrows \D:R$ be a partial Quillen on the subcategories $\C_0$, $\D_0$. Assume every object in $\C$ is cofibrant and every object in $\D$ is fibrant. Then the induced adjunction $L_*:\C^I_\proj \rightleftarrows \D^I_\proj:R_*$ is a partial Quillen equivalence on the subcategories $\widetilde{C_0}^I$, $\widetilde{\D_0}^I$.
\begin{proof} By \cref{lem-saturation}, $L\dashv R$ is a partial Quillen equivalence on the subcategories $\widetilde{\C_0}$ and $\widetilde{\D_0}$. In order to apply \cref{lem-sciarappa} to them, we need to check that $\widetilde{\C_0}^I$ and $\widetilde{\D_0}^I$ satisfy Condition (1) of \cref{def-partial}. It suffices to check these subcategories are closed under weak equivalences. This follows from weak equivalences in the projective model structure being component-wise,  since $\widetilde{\C_0}$ and $\widetilde{\D_0}$ are closed under weak equivalences.
    \end{proof}

\end{lemma}

\section{The cdgl models}
\label{sec:Lie}

We now turn to Lie models for $G$-equivariant spaces. We first review the main result from \cite{LieModelsInTopology}: we define the category of complete dg Lie algebras and we explain how it models rational spaces via a partial Quillen equivalence. We then extend this to the equivariant setting.

A Lie algebra object in $\Ch_*$ is called an (unbounded) \emph{differential graded Lie algebra}, or \emph{dgl} for short. 

\begin{definition} Let $L$ be %
a dgl. We say it is \emph{complete} 
      if it is moreover endowed with 
a descending filtration by differential Lie ideals
\[L=F^1L\supset \cdots \supset F^nL \supset F^{n+1}L\supset \cdots\] which is compatible with the bracket and for which the natural map
$L \to \varprojlim L/F^n L$ is a dgl isomorphism.

A dgl morphism $f:L\to L'$ between complete dg Lie algebras is a \emph{complete dg Lie morphism} if it is compatible  with the filtrations. The complete dg Lie algebras together with the complete dg Lie morphisms form a category denoted $\cDGL$. We will abbreviate ``complete dg Lie algebra'' to ``cdgl''.

A cdgl is called \emph{connected} if it is concentrated in non-negative degrees. %
\end{definition}

In \cite[Ch. 6]{LieModelsInTopology}, the authors define a cosimplicial cdgl $\mathfrak{L}_\bullet$. This induces in the standard way a ``nerve and realization'' adjunction $\mathcal{L}:\sSet \rightleftarrows \cDGL: \langle - \rangle$.

\begin{theorem}\cite[Ch. 8]{LieModelsInTopology}
\label{teo: adjunction sSet cdgl}
The right-transferred model structure on $\cDGL$ via the functor $\langle - \rangle:\cDGL\to \sSet$ exists and is cofibrantly generated. %
In particular, %
the adjunction
\[
\mathcal L : \sSet \rightleftarrows \cDGL : \langle - \rangle
\]
becomes a Quillen adjunction. Every object in $\cDGL$ is fibrant.
\end{theorem}

\begin{remark} There is an internal description of the weak equivalences and the fibrations in this model structure on $\cDGL$, i.e. one that does not pass via the functor $\langle - \rangle$. Fibrations are simply surjections in non-negative degrees, and weak equivalences can be characterized using Maurer--Cartan elements. On connected cdgls, the weak equivalences are exactly the quasi-isomorphisms. See \cite[Thm. 8.1, Rmk. 8.6]{LieModelsInTopology} for details.
\end{remark}

Following \cite[Def. 2.2]{LieCharacterizationBKcompletion}, a connected cdgl $L$ is \emph{homologically nilpotent} if $H_0(L)$ is a nilpotent Lie algebra and acts nilpotently on $H_*(L)$. %
We now observe that the Quillen adjunction in \cref{teo: adjunction sSet cdgl} is a partial Quillen equivalence:

\begin{theorem} \label{thm-cdgl-Quillen}
The Quillen adjunction $\mathcal L : \sSet \rightleftarrows \cDGL : \langle - \rangle$ is a partial Quillen equivalence on the reduced nilpotent rational simplicial sets and the homologically nilpotent connected cdgls.
\begin{proof} This is a rewording of \cite[Thms. 0.5 and 2.12]{LieCharacterizationBKcompletion} using the language of partial Quillen equivalences. %
\end{proof}
\end{theorem}

We now apply \cref{lem-sciarappa2} to obtain new cdgl models for equivariant rational spaces.

\begin{theorem} \label{cor-cdgl}
    There is a Quillen adjunction \[\mathcal{L}_*:\sSet^\OGop_\proj \rightleftarrows \cDGL^\OGop_\proj:\langle-\rangle_*.\] It is a partial Quillen equivalence on the full subcategories of $\OG$-indexed presheaves of
    connected nilpotent rational %
    simplicial sets and the $\OG$-indexed presheaves of homologically nilpotent %
    cdgls weakly equivalent to a connected one.
    \begin{proof}
        In order to apply \cref{lem-sciarappa2}, it suffices to observe that the closure of the subcategory of nilpotent reduced simplicial sets under weak equivalences is the subcategory of nilpotent connected simplicial sets, %
and the analogous statement for the subcategory on the cdgl side is obvious. %
    \end{proof}
\end{theorem}

\begin{remark}
A cdgl is weakly equivalent to a connected cdgl if and only if it has a unique gauge equivalence class of Maurer--Cartan elements and its twisted homology at any Maurer--Cartan representative vanishes in negative degrees.
This is a straightforward consequence of the characterization of the weak equivalences of 
cdgls in \cite[Theorem 8.1]{LieModelsInTopology}.
\end{remark}

\section{The cdga models} \label{sec:bg}

We now turn to cdga models. We first fix some notation: A (unital) commutative algebra object in $\Chz$ is called a \emph{commutative differential graded algebra},
or \emph{cdga} for short. We denote their category by $\CDGA$. 
We denote by $\Q$ the cdga which has $\Q$ concentrated in degree zero.

We first review Bousfield--Gugenheim's model structure on cdgas \cite{Bousfield--Gugenheim} which models connected nilpotent Kan complexes of finite type. %
We then explain how to generalize it to the equivariant case. %

\subsection{Bousfield--Gugenheim's model structure on \texorpdfstring{$\CDGA$}{CDGA}} 

We now review %
Bousfield--Gugenheim's \cite{Bousfield--Gugenheim} work, which provides cdga models for spaces.

\begin{proposition} \label{prop-modelos-cdga}
\cite[Section 4]{Bousfield--Gugenheim}, %
\cite[Section 6.2]{Fresse-Teichmuller2}, 
\cite[5.1]{white}. 
The right-transferred model structure on $\CDGA$ via the forgetful functor $\CDGA\to \Chz$ exists and is cofibrantly generated. 
Explicitly, the weak equivalences are the quasi-isomorphisms, 
the fibrations are the degreewise surjections, and the generating (acyclic) cofibrations 
are obtained by applying the free cdga functor to the generating (acyclic) cofibrations of $\Chz$.
\end{proposition}

Before we can state the theorem linking this model structure to that of Kan complexes, we need a couple of definitions:
\begin{definition} \label{def-cdga-conventions} Let $A$ be a cdga.
    \begin{itemize}
    \item We say $A$ is \emph{connected} if the unit $\Q\to A$ induces an isomorphism $\Q\cong A^0$.
        \item We say $A$ is \emph{cohomologically connected} if the unit $\Q\to A$ induces an isomorphism $\Q\cong H^0(A)$. %
        \item We say $A$ is \emph{of finite type} if it is of finite type as a cochain complex, i.e. $A^n$ is a finite-dimensional rational vector space for all $n\geq 0$.  %
        \item If $A$ is %
        cohomologically connected, we say it is \emph{weakly of finite type} %
        if there exists a Sullivan minimal model $N\stackrel{\sim}{\to} A$ %
        such that $N$ is of finite type.\footnote{It is this condition that \cite[Page 51]{Bousfield--Gugenheim} call \emph{finite $\Q$-type}. %
        }

    \end{itemize}
\end{definition}

\begin{remarks}
\begin{enumerate}
 \item
    If $A$ is a connected cdga, then it is cohomologically connected. 
    Indeed, the unit $\eta:\Q\to A$ being a cochain map forces the differential $d:A^0\to A^1$ to be zero, so the isomorphism $\eta^0:\Q\to A^0$ induces an isomorphism $\Q \cong H^0(A)$. 
       In fact, in $\Ho(\CDGA)$ the full subcategories on the %
    cdgas which are connected or cohomologically connected are equivalent.  This is 
    essentially because their minimal models are always connected, and they exist for 
    every cohomologically connected cdga. %
        \item If two cdgas are cohomologically connected and quasi-isomorphic, then if one of them is weakly of finite type, so is the other. %
\end{enumerate}
\end{remarks}

We can now state the fundamental theorem of rational homotopy theory à la Sullivan %
in the version by Bousfield and Gugenheim. There is a simplicial cdga of \emph{piecewise linear differential forms} which induces in the standard way a ``nerve and realization'' adjunction $\sSet \rightleftarrows \CDGA^\op$, where the left adjoint is the functor of piecewise linear differential forms $\APL$ and the right adjoint is denoted by $\langle - \rangle$.

\begin{theorem} 
\label{thm-fund} 
\cite[Thm. 10.1]{Bousfield--Gugenheim} %
There is a Quillen adjunction \[\APL: \sSet\rightleftarrows \CDGA^\op:\langle - \rangle.\] 
It is a partial Quillen equivalence on the connected,
nilpotent rational %
simplicial sets of finite type and the cohomologically connected %
cdgas weakly of finite type.
\begin{proof}
Bousfield and Gugenheim do not word the result above as being a partial Quillen equivalence, but it is indeed what they prove. We note that both the subcategory of connected, nilpotent rational simplicial sets and that of cohomologically connected cdgas weakly of finite type are closed under weak equivalences. %
\end{proof}
\end{theorem}

\begin{remark} %
The broad class of ``cohomologically connected cofibrant cdgas weakly of finite type'' can be more explicitly described. 
For example, Fresse \cite[Sect. 7.3.7]{Fresse-Teichmuller2} introduces the class of \emph{nilpotent cell complexes of finite type}. In Thm. 7.3.10 he recovers the known result that one can construct a cofibrant replacement of this form for the cdga associated to a nilpotent, rational Kan
complex of finite type.
\end{remark}

\subsection{The injective model structure on \texorpdfstring{$\CDGA^\OG$}{Fun(OG,CDGA)}} %
We now apply \cref{lem-sciarappa} to cdga models, as observed in \cite{Sciarappa2017}. The contravariance of the cdga models means a bit of care is needed. %

\begin{lemma} \label{lem-cdga-combinatorial} $\CDGA$ is a combinatorial model category.
\begin{proof} We already observed $\CDGA$ is cofibrantly generated in \cref{prop-modelos-cdga}. %
To see it is locally presentable, by \cite[2.78]{adamek-rosicky}, it suffices to see that $\Chz$ is locally presentable and that the functor $\Chz\to \Chz$ defined as the composition of the free functor $\Chz\to \CDGA$ with the forgetful functor preserves filtered colimits. %
These two facts are true: see \cite[Prop. 3.10]{Beke} and \cite[Prop. 2.4.1]{Fresse-2009}. 
\end{proof}
\end{lemma}

Call a functor $\OG\to \CDGA$ a \emph{$G$-system of cdgas}.

 \begin{theorem} \label{cor-sciarappa-cdga}
There is a  Quillen adjunction \[\left(\APL\right)_*:\sSet^\OGop_\proj\rightleftarrows \left(\CDGA^\OG_ \inj\right)^\op:\langle - \rangle_*,\] 
where $\proj$ and $\inj$ denote the projective and injective model structures, respectively. It is a partial Quillen equivalence on the $\OG$-indexed presheaves of connected,
nilpotent rational %
simplicial sets of finite type and the $G$-systems of cohomologically connected %
cdgas weakly of finite type.
 \end{theorem}

\begin{proof}
Consider the Quillen adjunction $\APL:\sSet\rightleftarrows \CDGA^\op:\langle - \rangle$ of \cref{thm-fund}. We want to apply  \cref{lem-sciarappa} to it, for $I=\OGop$.  By \cref{rmk-inheritance}, since the subcategories of $\sSet$ and $\CDGA^\op$ under consideration are closed under weak equivalences, it only remains to verify that the projective model structures exist.

 The projective model structure on $\sSet^\OGop$ exists because $\sSet$ is cofibrantly generated.  We claim that the projective model structure on $(\CDGA^\op)^\OGop$ also exists. This is not as straightforward because the target model category has the opposite model structure to a cofibrantly generated one, so a priori it has no reason to be cofibrantly generated as well. On the other hand, we can use the following trick: for any category $\C$ and any small category $I$ there is an isomorphism of categories 
\[\Fun(I^\op, \C^\op) \cong \Fun(I, \C)^\op.\] %
By definition of the projective, injective and opposite model structures, it follows that the projective model structure on $\Fun(\OGop, \C^\op)$ exists if and only if the injective model structure on $\Fun(\OG, \C)$ exists, in which case $\Fun(\OGop, \C^\op)_\proj \cong \left(\Fun(\OG, \C)_\inj\right)^\op$.  In our situation, we have $\C=\CDGA$, which we just proved is combinatorial, so the injective model structure exists.

Applying \cref{lem-adjunction-on-fun} and the above observations, we get a Quillen adjunction followed by an isomorphism:
\[\sSet^\OGop_\proj\rightleftarrows (\CDGA^\op)^\OGop_\proj \cong  \left(\CDGA^\OG_\inj\right)^\op.\]
We can now apply \cref{lem-sciarappa} to \cref{thm-fund} as we wanted, and we get the result.
\qedhere

\end{proof}

This gives an algebraic model for the genuine equivariant homotopy theory of connected, nilpotent, finite-type rational Kan complexes.

\begin{remark}
    The proofs of both \cref{cor-sciarappa-cdga} and \cref{cor-cdgl} do not use any properties of $\OG$, so in their statements one could replace it by any small category. On the other hand, the presence of $\OG$ is what allows us to interpret $\sSet^\OGop_\proj$ as a model for the genuine equivariant homotopy theory of spaces, thanks to \cref{thm-elmendorf}.
\end{remark}

\subsection{The \texorpdfstring{$r(\inj)$}{r(inj)} model structure on \texorpdfstring{$\CDGA^\OG$}{Fun(OG,CDGA)}} 
We will now describe a different model structure on $\CDGA^\OG$. It will be needed for the connection to the equivariant minimal models of \cite{Triantafillou} as cofibrant replacements, after \cite{Scull-2008}.  More comments on the relationship of our work with Scull's will be made in \cref{sec-scull}. We will need to work out three different model structures: since the work is quite long, we relegate the details and proofs to \cref{sect-two-model} for ease of reading.

Denote by $\Vect$ the category of rational vector spaces. Consider $\Vect^\OG$: since $\Vect$ is an abelian category, then $\Vect^\OG$ is so, too. We can therefore consider $\Chz(\Vect^\OG)$, the category of non-negative cochain complexes in it. %

We will use the following injective model structure on $\Chz(\Vect^\OG)$. Its existence follows from  \cite[4.4]{Bousfield--CosimplicialResolutions}. We will review it in detail in \cref{sect-model-chz}, where we also prove it is cofibrantly generated by giving explicit generating (acyclic) cofibrations.  %
We use the following terminology: an object $A=\{A^n(H)\}$ in $\Chz(\Vect^\OG)$ has \emph{degrees} $n\geq 0$ and \emph{components} $H\in \OG$.

\begin{theorem} \label{thm-model-ChzVectOG} There is a model structure on $\Chz(\Vect^\OG)$ whose weak equivalences are the quasi-isomorphisms, the cofibrations are the component-wise degreewise injections in positive degrees, and the fibrations are the degreewise component-wise surjections with degreewise injective kernel.
\end{theorem}

\begin{remark}
 By ``injective kernel'' we mean injective in the sense of homological algebra: the kernel of a fibration $f:A\to B$ is a cochain complex $\ker(f)\in \Chz(\Vect^\OG)$ and we ask that %
$\ker(f)^n\in \Vect^\OG$ be injective for all $n\geq 0$. %
\end{remark}

\begin{remark} \label{rmk-inj-model-st} There is a straightforward isomorphism of categories \[\Chz(\Vect^\OG) \cong (\Chz)^\OG. \]
Under this isomorphism, the model structure from \cref{thm-model-ChzVectOG} coincides with the injective model structure on the functor category $(\Chz)^\OG$, i.e. the model structure where a natural transformation is a weak equivalence (resp. cofibration) if and only if it is so component-wise.
\end{remark}

We will now transfer the above model structure to $G$-systems of cdgas. Let $F:\Chz \rightleftarrows \CDGA:U$ be the free--forgetful adjunction. Let \[F_*:(\Chz)^\OG \rightleftarrows \CDGA^\OG:U_*\] be the induced adjunction by post-composition. Endow $(\Chz)^\OG$ with the injective model structure (see \cref{rmk-inj-model-st}). We call the right-transferred model structure on $\CDGA^\OG$ along $U_*$ the \emph{$r(\inj)$ model structure} on $\CDGA^\OG$. We will prove the following theorem in \cref{sect-rinj}, where we also prove this structure is cofibrantly generated and give concrete generating (acyclic) cofibrations.

\begin{theorem}
The $r(\inj)$ model structure on $\CDGA^\OG$ exists. Its weak equivalences are the component-wise weak equivalences and its fibrations are the degreewise component-wise surjections with degreewise injective kernel as an object of $\Vect^\OG$.
\end{theorem}

\begin{remark} \label{rmk-inj-r(inj)-coincide}
    When $G$ is the trivial group, both the injective and the $r(\inj)$ model structures on $\CDGA^\OG$ coincide with Bousfield--Gugenheim's model structure on $\CDGA$ reviewed in \cref{prop-modelos-cdga}.
\end{remark}

We now relate the $r(\inj)$ model structure on $\CDGA^\OG$ to the injective one:

\begin{proposition} \label{prop-comparison-inj-r(inj)} There is a Quillen equivalence
\[\begin{tikzcd}
	{\CDGA^\OG_{r(\inj)}} & {\CDGA^\OG_\inj.}
	\arrow[""{name=0, anchor=center, inner sep=0}, "\id", shift left=2, from=1-1, to=1-2]
	\arrow[""{name=1, anchor=center, inner sep=0}, "\id", shift left=2, from=1-2, to=1-1]
	\arrow["\dashv"{anchor=center, rotate=-90}, draw=none, from=0, to=1]
\end{tikzcd}\]
\begin{proof}
It suffices to prove that a fibration $f$ in $\CDGA^\OG_\inj$ is in particular a fibration in $\CDGA^\OG_{r(\inj)}$. Indeed, we will then have that $\id:\CDGA^\OG_\inj \to \CDGA^\OG_{r(\inj)}$ preserves (acyclic) fibrations, so the adjunction is a Quillen adjunction, and therefore a Quillen equivalence.

The forgetful functor $U:\CDGA\to \Chz$ is right Quillen, so
$U_*:\CDGA^\OG_\inj\to (\Chz)^\OG_\inj\cong \Chz(\Vect^\OG)$ is right Quillen by \cref{lem-adjunction-on-fun}. Hence $U_*f$ is a fibration in $\Chz(\Vect^\OG)$, whose fibrations are the epimorphisms with degreewise injective kernel (\cref{thm-model-ChzVectOG}).
\end{proof}
\end{proposition}

The following theorem is the $r(\inj)$ analogue to \cref{cor-sciarappa-cdga} which used the injective model structure. 

 \begin{theorem} \label{thm-spaces-r(inj)}
There is a  Quillen adjunction \[\left(\APL\right)_*:\sSet^\OGop_\proj\rightleftarrows \left(\CDGA^\OG_{r(\inj)}\right)^\op:\langle - \rangle_*.\] 
It is a partial Quillen equivalence on the $\OG$-indexed presheaves of connected,
nilpotent rational %
simplicial sets of finite type and the $G$-systems of cohomologically connected %
cdgas weakly of finite type.
\begin{proof}
We use the following simple observation: if $L:\C\rightleftarrows \D:R$ is a Quillen equivalence, then the opposite adjunction $R^\op:\D^\op \rightleftarrows \C^\op:L^\op$ is also a Quillen equivalence. We apply this to the Quillen equivalence in \cref{prop-comparison-inj-r(inj)}  and compose it with \cref{cor-sciarappa-cdga} by means of \cref{lem-composition-partial} to get the result:
\[\begin{tikzcd}
	{\sSet^{\OGop}_\proj} & {\left(\CDGA^\OG_\inj\right)^\op} & {\left(\CDGA^\OG_{r(\inj)}\right)^\op.}
	\arrow[""{name=0, anchor=center, inner sep=0}, "{(\APL)_*}", from=1-1, to=1-2]
	\arrow[""{name=1, anchor=center, inner sep=0}, "{\langle - \rangle_*}", shift left=3, from=1-2, to=1-1]
	\arrow[""{name=2, anchor=center, inner sep=0}, "\id", from=1-2, to=1-3]
	\arrow[""{name=3, anchor=center, inner sep=0}, "\id", shift left=3, from=1-3, to=1-2]
	\arrow["\dashv"{anchor=center, rotate=-90}, draw=none, from=0, to=1]
	\arrow["\dashv"{anchor=center, rotate=-90}, draw=none, from=2, to=3]
\end{tikzcd} \qedhere\]
\end{proof}
 \end{theorem}

\begin{remark} \label{rmk-pointed-variant}
Let us give a sketch of how one can get a pointed %
variant of the above theorem. Replace the Bousfield--Gugenheim partial Quillen equivalence with the analogous one $\sSet_* \rightleftarrows \CDGA_{/\Q}^\op$, where $\CDGA_{/\Q}$ denotes the category of augmented cdgas, as already worked out in \cite[Thm. 10.1]{Bousfield--Gugenheim}. 
We can consider the augmented (in the sense of \cref{prop-overunder}) injective model structure on $\CDGA^\OG_{/\Q}$, which can be described either as $(\CDGA^\OG_\inj)_{/\Q}$ or as $(\CDGA_{/\Q})^\OG)_\inj$, or the augmented $r(\inj)$ model structure, as we explain in \cref{rmk-model-cdga1-via-aug}. Then the augmented analog of \cref{prop-comparison-inj-r(inj)} %
holds \emph{mutatis mutandis}, and thus the pointed/augmented analog of \cref{thm-spaces-r(inj)} also holds.
\end{remark}

\section{Three model structures}  \label{sect-two-model}

We now turn to the proofs of the model-categorical statements we made in the previous section. The first part  concerns the cofibrant generation of the injective model structure on $\Chz(\Vect^\OG)$. The second part concerns the existence and cofibrant generation of the $r(\inj)$ model structure on $\CDGA^\OG$. The third part does the same but for connected cdgas: this has not appeared above but will be used in \cref{sec-other-models}.

\subsection{The injective model structure on \texorpdfstring{$\Chz(\Vect^\OG)$}{Ch(Fun(Vect,OG))}} \label{sect-model-chz}

We will now describe the injective model structure on $\Chz(\Vect^\OG)$. Its existence is well-known. Already Quillen \cite{Quillen-HomotopicalAlgebra} established the existence of the \emph{projective} model structure on $\Chaz(\A)$ where $\A$ is an abelian category with enough projectives, using the Dold--Kan correspondence with simplicial objects. Dually, Bousfield \cite[4.4]{Bousfield--CosimplicialResolutions} established the existence of the injective model structure on $\Chz(\A)$ where $\A$ is an abelian category with enough injectives, via the Dold--Kan correspondence with cosimplicial objects. This is the structure we will adopt. However, Bousfield did not address the question of whether this model structure is cofibrantly generated. We will now prove it is. We will concentrate on the case where $\A=\Vect^\OG$, since it is the case that interests us, but we will later indicate how the generating (acyclic) cofibrations that we consider generalize to other abelian categories.

In order to specify what our sets of generating (acyclic) cofibrations in $\Chz(\Vect^\OG)$ will be, we need some preliminary considerations.

\subsubsection{A Grothendieck abelian category}
Recall that if $i:\C_0\to \C$ is %
a functor and $\V$ is a %
cocomplete category, then the precomposition functor $i^*:\Fun(\C,\V)\to \Fun(\C_0, \V)$ has a left %
adjoint given by left %
Kan extension along $i$, and it has an explicit %
expression as a colimit \cite[Cor. 6.2.6]{Riehl-Context}.

Now, for an $H\in \OG$ we can consider $\Aut_{\OG}(H)\cong N_G(H)/H=:W(H)$, the Weyl group of $H$. %
This generates the one-object %
subcategory $BW(H)$ of $\OG$. Take $i$ to be its inclusion into $\OG$ and $\V$ to be $\Vect$: the functor $i^*$ becomes $\ev_H:\Vect^\OG \to \Vect^{BW(H)}\cong \QWH$, where $\QWH$ denotes the category of left modules over the group ring $\Q[W(H)]$. We let $P_H$ %
denote the left %
Kan extension functor along $i$, so we have %
an adjunction
\[\begin{tikzcd}
	\QWH && {\Vect^\OG.}
	\arrow[""{name=0, anchor=center, inner sep=0}, "{P_H}", shift left=2, from=1-1, to=1-3]
	\arrow[""{name=1, anchor=center, inner sep=0}, "{\ev_H}", shift left=2, from=1-3, to=1-1]
	\arrow["\dashv"{anchor=center, rotate=-90}, draw=none, from=0, to=1]
\end{tikzcd}\]
Explicitly, from the colimit formula for the Kan extension we get, for $V\in \QWH$ and $K\in \OG$:
\[P_H(V)(K)=\Q[\Hom_\OG(H,K)] \otimes_{\Q[W(H)]} V\in \Vect.\]

The objects $P_H(V)\in \Vect^\OG$ are projective%
, since $\ev_H$ is exact (it preserves epimorphisms) so its left adjoint $P_H$ preserves projectives. %
We use them to define a set of (projective) generators in this category. Recall the following simple categorical observation: if %
$\begin{tikzcd}
	{\A_i} & \C
	\arrow[""{name=0, anchor=center, inner sep=0}, "{L_i}", from=1-1, to=1-2]
	\arrow[""{name=1, anchor=center, inner sep=0}, "{R_i}", shift left=3, from=1-2, to=1-1]
	\arrow["\dashv"{anchor=center, rotate=-90}, draw=none, from=0, to=1]
\end{tikzcd}$
is a set of adjunctions, $Q_i$ is a generator of $\A_i$ and the $R_i$ are jointly faithful %
(meaning: two morphisms $f$ and $g$ in $\C$ are equal as soon as $R_i(f)=R_i(g)$ for all $i$), then the $L_iQ_i$ form a set of generators of $\C$.

Now take $P_H\dashv \ev_H$, $H\in \OG$ as our adjunctions: since the regular module $\Q[W(H)]$ is a generator of $\QWH$, we deduce that $\{P_H(\Q[W(H)])\}_{H\in \OG}$ is a set of generators in $\Vect^\OG$. We introduce the notation
\[\PH:=P_H(\Q[W(H)])\in \Vect^\OG.\]
More explicitly, $\PH(K)=\Q[\Hom_\OG(H,K)]$. %
We remark that the direct sum $\bigoplus_{H\in \OG} \PH$ is a generator in $\Vect^\OG$ (in any cocomplete abelian category, the direct sum of a set of generators is a generator).

Finally, note that $\Vect^\OG$ is a Grothendieck abelian category, i.e. it is cocomplete, it has a generator and filtered colimits are exact (since colimits are computed pointwise).\footnote{More generally, a functor category $\A^I$ is Grothendieck abelian if $\A$ is.} %

\subsubsection{Sphere and disk functors} \label{subsub-spheres} For $n\geq 0$ we have a functor $S^n:\Vect^\OG\to \Chz(\Vect^\OG)$ which concentrates in degree $n$. It is left adjoint to the cycles functor $Z^n:\Chz(\Vect^\OG)\to \Vect^\OG$. For $n\geq 1$ we have a functor $D^n:\Vect^\OG\to \Chz(\Vect^\OG)$ which places a copy of an object in $\Vect^\OG$ in degrees $n-1$ and $n$, with the identity between them being the unique non-zero differential. It is left adjoint to the evaluation functor $\ev_{n-1}:\Chz(\Vect^\OG)\to \Vect^\OG$. These functors are related by natural transformations $j: S^n\Rightarrow D^n$ (the inclusion) and $d: \ev_{n-1}\Rightarrow Z^n$ (the differential):
\[\begin{tikzcd}
	{\Vect^\OG} & { \Chz(\Vect^\OG)} && {\Vect^\OG} & { \Chz(\Vect^\OG)} && \begin{array}{c} j: S^n\Rightarrow D^n \\ d:\ev_{n-1}\Rightarrow Z^n \end{array}
	\arrow[""{name=0, anchor=center, inner sep=0}, "{S^n}", from=1-1, to=1-2]
	\arrow[""{name=1, anchor=center, inner sep=0}, "{Z^n}", shift left=3, from=1-2, to=1-1]
	\arrow[""{name=2, anchor=center, inner sep=0}, "{D^n}", from=1-4, to=1-5]
	\arrow[""{name=3, anchor=center, inner sep=0}, "{\ev_{n-1}}", shift left=3, from=1-5, to=1-4]
	\arrow["\dashv"{anchor=center, rotate=-90}, draw=none, from=0, to=1]
	\arrow["\dashv"{anchor=center, rotate=-90}, draw=none, from=2, to=3]
\end{tikzcd}\]
What we will say next can be checked by hand; more abstractly, it follows from %
these two natural transformations being \emph{conjugate}, in the sense of \cite[Sect. IV.7]{MacLane}. For any $\delta:U\to V$ in $\Vect^\OG$, we have that %
\begin{equation} \label{eq-spheredisk}%
\begin{tikzcd}
	{S^nU} & X \\
	{D^nV} & Y
	\arrow["\alpha", from=1-1, to=1-2]
	\arrow["{D^n(\delta) \circ j_U}"', from=1-1, to=2-1]
	\arrow["f", from=1-2, to=2-2]
	\arrow["\beta"', from=2-1, to=2-2]
\end{tikzcd} \quad \text{commutes} \quad \iff \quad %
\begin{tikzcd}
	U & {Z^nX} \\
	V & {Z^nY}
	\arrow["a", from=1-1, to=1-2]
	\arrow["\delta"', from=1-1, to=2-1]
	\arrow["Z^n(f)", from=1-2, to=2-2]
	\arrow["{d_Y \circ b}"', from=2-1, to=2-2]
\end{tikzcd} \quad \text{commutes,}\end{equation}
where $a$ is the transpose of $\alpha$ and $b$ is the transpose of $\beta$.
Moreover, lifts in the diagram on the left are in one-to-one correspondence with maps $h: V\to X^{n-1}$ such that the following two diagrams commute:
\begin{equation}\label{eq-spheredisk-lift}\begin{tikzcd}
	& {X^{n-1}} \\
	V & {Y^{n-1}}
	\arrow["{f^{n-1}}", from=1-2, to=2-2]
	\arrow["h", from=2-1, to=1-2]
	\arrow["b"', from=2-1, to=2-2]
\end{tikzcd}
\quad \quad
\begin{tikzcd}
	U & {Z^nX} \\
	V & {X^{n-1}}
	\arrow["a", from=1-1, to=1-2]
	\arrow["\delta"', from=1-1, to=2-1]
	\arrow["h"', from=2-1, to=2-2]
	\arrow["d"', from=2-2, to=1-2]
\end{tikzcd}\end{equation}

\subsubsection{The case of \texorpdfstring{$\id_{\PH}$}{id_PH}}\label{subsub-idPH} We will use the following particular case of (\ref{eq-spheredisk}) several times: when we take $\delta$ to be $\id:\PH\to \PH$, then commutativity of the following three diagrams is equivalent:
\[\begin{tikzcd}
	{S^n(\PH)} & X \\
	{D^n(\PH)} & Y
	\arrow["\alpha", from=1-1, to=1-2]
	\arrow[from=1-1, to=2-1]
	\arrow["f", from=1-2, to=2-2]
	\arrow["\beta"', from=2-1, to=2-2]
\end{tikzcd} \quad \iff \quad %
\begin{tikzcd}
	\PH & {Z^nX} \\
	\PH & {Z^nY}
	\arrow["a", from=1-1, to=1-2]
	\arrow["\id"', from=1-1, to=2-1]
	\arrow["{Z^n(f)}", from=1-2, to=2-2]
	\arrow["{d\circ b}"', from=2-1, to=2-2]
\end{tikzcd}
\quad \iff \quad
\begin{tikzcd}
	{\Q[W(H)]} & {Z^nX(H)} \\
	{\Q[W(H)]} & {Z^nY(H)}
	\arrow["a", from=1-1, to=1-2]
	\arrow["\id"', from=1-1, to=2-1]
	\arrow["{Z^nf(H)}", from=1-2, to=2-2]
	\arrow["{d\circ b}"', from=2-1, to=2-2]
\end{tikzcd}\]
The first equivalence is a particular case of the previous point. The last diagram takes place in $\QWH$, and its commutativity is equivalent to that of the middle diagram because of the adjunction $P_H \dashv \ev_H$ (by abuse of notation we have denoted the transposes of $a$ and $b$ by $a$ and $b$). In $\QWH$, a map from $\Q[W(H)]$ amounts to choosing an element, %
so commutativity of the last diagram is equivalent to having elements \[a\in Z^nX(H), \quad b\in Y^{n-1}(H) \quad \text{ such that } \quad fa=db, \] and a lift in the diagram on the left is equivalent to having an element \[h\in X^{n-1}(H) \quad \text{ such that } \quad fh=b  \quad \text{ and } \quad dh=a.\]

\subsubsection{Baer sets} \label{subsub-baer} In an abelian category $\A$, define a \emph{Baer set} to be a set of monomorphisms \[\{\delta_t:U_t\to V_t\}_{t\in T}\] that detect injectivity, in the sense that an object $A\in \A$ is injective if and only if every map $U_t\to A$ admits an extension to a map $V_t \to A$.

In a Grothendieck abelian category, the inclusions of subobjects of a generator $P$ into $P$ form a Baer set \cite[Lem. 1, Sect. 1.10]{Grothendieck-Tohoku}. Alternatively, in \cite[Lem. 6.4]{Scull-2008}, Scull provides a different Baer set for $\Vect^\OG$ assuming $G$ is finite. %

We let $\{\delta_t:U_t\to V_t\}_{t\in T}$ denote any Baer set in $\Vect^\OG$. %

\subsubsection{The sets \texorpdfstring{$I$}{I} and \texorpdfstring{$J$}{J}} \label{subsub-IJ} We can finally define the sets of maps in $\Chz(\Vect^\OG)$ which will be our generating (acyclic) cofibrations. Let
\[J=\left\{0\to D^n(\PH)\right\}_{\substack{n\geq 1 \\ H\in \OG}} \quad \cup \quad \left\{D^n(U_t)\to D^n(V_t)\right\}_{\substack{n\geq 1 \\ t\in T}}.\]
In other words, $J$ is the set of maps $\{D^n(f)\}_{n\geq 1}$ where $f$ is either $0 \to \PH$ for some $H\in \OG$ or $U_t\to V_t$ for some $t\in T$.

Let
\[I=\left\{S^n(\PH) \to D^n(\PH)\right\}_{\substack{n\geq 1 \\ H \in \OG}} \quad \cup \quad \left\{S^0(\PH) \to 0 \right\}_{H\in \OG} \quad \cup \quad \left\{S^n(U_t) \to D^n(V_t)\right\}_{\substack{n\geq 1 \\ t\in T}}.\]
We can also reinterpret this set. For an $f:U\to V$ in $\Vect^\OG$ and $n\geq 1$, let $\gamma^n(f):S^n(U)\to D^n(V)$ be $D^n(f) \circ j_U$, as in (\ref{eq-spheredisk}). Then $I$ is the set of maps $\{\gamma^n(f)\}_{n\geq 1}$ where $f$ is either $\id:\PH\to \PH$ or $\delta_t:U_t\to V_t$, plus the maps $S^0(\PH)\to 0$. We can understand the latter to be of the form $\gamma^0(\id_{\PH})$ if we define $D^{-1}$ to be the zero functor.

\subsubsection{The model structure}

We now establish the existence of the desired cofibrantly generated model structure on this category, where the weak equivalences are the quasi-isomorphisms, the fibrations are the epimorphisms with degreewise injective kernel, and the cofibrations are the monomorphisms in positive degrees. More explicitly:

\begin{theorem} \label{thm-IJ}
    There is an \emph{injective} model structure on $\Chz(\Vect^\OG)$ whose weak equivalences are the quasi-isomorphisms, the cofibrations are the injections component-wise in positive degrees, and the fibrations are the degreewise component-wise surjections with degreewise injective kernel. 

    It is cofibrantly generated, with generating cofibrations $I$ and generating acyclic cofibrations $J$ as defined above. Every object is cofibrant, and the fibrant objects are the complexes which are degreewise injective.
\end{theorem}

\begin{proof} %
The existence of the structure is furnished by \cite[4.4]{Bousfield--CosimplicialResolutions}. The only hypothesis that needs to be checked is that $\Vect^\OG$ has enough injectives. But this is true: either by direct construction \cite[Sect. 4]{Triantafillou}, or by abstract reasons. Indeed, $\Vect^\OG$ is Grothendieck abelian, so it has enough injectives \cite[Thm. 1.10.1]{Grothendieck-Tohoku}.

Since $\Chz(\Vect^\OG)$ is a Grothendieck abelian category, it is locally presentable \cite[Prop. 3.10]{Beke}, so every object is small with respect to all morphisms %
    and therefore $I$ and $J$ permit the small object argument.

Let $f:X\to Y$ in $\Chz(\Vect^\OG)$ and $K=\ker(f)\in \Vect^\OG$. We will now prove that $J^\boxslash$ is the class of  fibrations.

\begin{itemize} 
\item \underline{$f$ epimorphism equivalent to lifts against $\{0\to D^n(\PH)\}_{n\geq 1, H\in \OG}$}:

Let $n\geq 1$, $H\in \OG$ and $b\in Y^{n-1}(H)$. To find an $u\in X^{n-1}(H)$ lifting $b$ means to find a lift in the following leftmost commutative diagram in $\QWH$.  Using the $P_H\dashv \ev_H$ then the $D^n\dashv \ev_{n-1}$ adjunctions, this is equivalent to lifting the transpose $\beta$ along $f$ against a map $0\to D^n(\PH)$:

\[\begin{tikzcd}
	0 & {X^{n-1}(H)} \\
	{\Q[W(H)]} & {Y^{n-1}(H)}
	\arrow[from=1-1, to=1-2]
	\arrow[from=1-1, to=2-1]
	\arrow["f", from=1-2, to=2-2]
	\arrow["u", dotted, from=2-1, to=1-2]
	\arrow["b"', from=2-1, to=2-2]
\end{tikzcd}
\quad \iff \quad
\begin{tikzcd}
	{0=P_H(0)} & {X^{n-1}} \\
	\PH & {Y^{n-1}}
	\arrow[from=1-1, to=1-2]
	\arrow[from=1-1, to=2-1]
	\arrow[from=1-2, to=2-2]
	\arrow[dotted, from=2-1, to=1-2]
	\arrow[from=2-1, to=2-2]
\end{tikzcd}
\quad \iff \quad%
\begin{tikzcd}
	{0=D^n(0)} & X \\
	{D^n(\PH)} & Y
	\arrow[from=1-1, to=1-2]
	\arrow[from=1-1, to=2-1]
	\arrow["f", from=1-2, to=2-2]
	\arrow[dotted, from=2-1, to=1-2]
	\arrow["\beta"', from=2-1, to=2-2]
\end{tikzcd}\]

\item 
\underline{Lifts against $\{D^n(U_t)\to D^n(V_t)\}_{n\geq 1, t\in T}$ imply degreewise injectivity of $K$:} %

Assume we have lifts against $D^n(\delta_t):D^n(U_t)\to D^n(V_t)$ for all $n\geq 1$ and $t\in T$. Let $n\geq 0$. We want to prove that $K^n\in \Vect^\OG$ is injective. It suffices to solve the following  extension problem, for an arbitrary $\phi:U_t\to K^n$ and $t\in T$:
\[\begin{tikzcd}
	{U_t} & {V_t} \\
	{K^n}
	\arrow["{\delta_t}", from=1-1, to=1-2]
	\arrow["\phi"', from=1-1, to=2-1]
	\arrow[dotted, from=1-2, to=2-1]
\end{tikzcd}\]
We reformulate this lifting problem as finding a lift in the following diagram on the left, where $i:K^n\hookrightarrow X^n$ is the inclusion. Indeed, $fu=0$ means $u$ lands in $K^n$, and $i \phi=u\delta_t$ means $u$ is the desired extension. But finding such  a lift  is equivalent to finding one in the diagram on the right by the $D^{n+1}\dashv \ev_n$ adjunction, and this lift exists by hypothesis:
\[\begin{tikzcd}
	{U_t} & {X^n} \\
	{V_t} & {Y^n}
	\arrow["{i \circ \phi}", from=1-1, to=1-2]
	\arrow["{\delta_t}"', from=1-1, to=2-1]
	\arrow["f", from=1-2, to=2-2]
	\arrow["u", dotted, from=2-1, to=1-2]
	\arrow["0"', from=2-1, to=2-2]
\end{tikzcd}
\quad \iff \quad
\begin{tikzcd}
	{D^{n+1}(U_t)} & X \\
	{D^{n+1}(V_t)} & Y
	\arrow[from=1-1, to=1-2]
	\arrow["{D^{n+1}(\delta_t)}"', from=1-1, to=2-1]
	\arrow["f", from=1-2, to=2-2]
	\arrow[dotted, from=2-1, to=1-2]
	\arrow[from=2-1, to=2-2]
\end{tikzcd}\]

\item \underline{$f$ epimorphism with degreewise injective $K$ implies lifts against $\{D^n(U_t)\to D^n(V_t)\}_{n\geq 1, t\in T}$:} 

Assume $f$ is an epimorphism and $K$ is degreewise injective. We want to find a lift in the following diagram on the left, for all $n\geq 1$ and $t\in T$. By the $D^n\dashv \ev_{n-1}$ adjunction, this is equivalent to finding a lift in the diagram on the right:
\[\begin{tikzcd}
	{D^n(U_t)} & X \\
	{D^n(V_t)} & Y
	\arrow["\alpha", from=1-1, to=1-2]
	\arrow[from=1-1, to=2-1]
	\arrow["f", from=1-2, to=2-2]
	\arrow[dotted, from=2-1, to=1-2]
	\arrow["\beta"', from=2-1, to=2-2]
\end{tikzcd}
\quad \iff \quad
\begin{tikzcd}
	{U_t} & {X^{n-1}} \\
	{V_t} & {Y^{n-1}}
	\arrow["a", from=1-1, to=1-2]
	\arrow["{\delta_t}"', from=1-1, to=2-1]
	\arrow["f", from=1-2, to=2-2]
	\arrow["b"', from=2-1, to=2-2]
    \arrow[dotted, from=2-1, to=1-2]
\end{tikzcd}\]
Now, since $f$ is an epimorphism, there is a short exact sequence $0\to K^{n-1} \hookrightarrow X^{n-1} \stackrel{f}{\to}Y^{n-1}\to 0$ and it splits because $K^{n-1}$ is injective. Therefore, there exists a morphism $s: Y^{n-1}\to X^{n-1}$ such that $fs=\id$. Consider the map $a-sb\delta_t:U_t\to X^{n-1}$. We have
\[f(a-sb\delta_t)=fa -fsb\delta_t=fa-b\delta_t=0,\]
so $a-sb\delta_t$ lands in $K^{n-1}$ which is injective, so the following dotted extension exists:
\[\begin{tikzcd}
	{U_t} & {V_t} \\
	{K^{n-1}}
	\arrow["{\delta_t}", from=1-1, to=1-2]
	\arrow["{a-sb\delta_t}"', from=1-1, to=2-1]
	\arrow["p", dotted, from=1-2, to=2-1]
\end{tikzcd}\]
Now consider the map $sb+p:V_t\to X^{n-1}$. It is the desired lift. Indeed, $f(sb+p)=fsb+fp=b$, and \[(sb+p)\delta_t=sb\delta_t+p\delta_t=sb\delta_t+a-sb\delta_t=a.\]
\end{itemize}
\smallskip

This proves $J^\boxslash$ is the class of fibrations. Now we will prove that $I^\boxslash$ is the class of acyclic fibrations, i.e. epimorphisms which are quasi-isomorphisms, or, equivalently by the long exact sequence in cohomology, epimorphisms with degreewise injective \emph{acyclic} kernel. Call $\FW$ this class of morphisms. \smallskip
\begin{itemize}
\item \underline{$H^0(f)$ monomorphism equivalent to lifts against $\{S^0(\PH)\to 0\}_{H\in \OG}$: }

Commutativity of each of the following diagrams is equivalent to commutativity of the other, where the last one takes place in $\QWH$:
\[\begin{tikzcd}
	{S^0(\PH)} & X \\
	0 & Y
	\arrow[from=1-1, to=1-2]
	\arrow[from=1-1, to=2-1]
	\arrow["f", from=1-2, to=2-2]
	\arrow[dotted, from=2-1, to=1-2]
	\arrow[from=2-1, to=2-2]
\end{tikzcd}
\quad \iff \quad
\begin{tikzcd}
	\PH & {Z^0X} \\
	0 & {Z^0Y}
	\arrow[from=1-1, to=1-2]
	\arrow[from=1-1, to=2-1]
	\arrow["f", from=1-2, to=2-2]
	\arrow[dotted, from=2-1, to=1-2]
	\arrow[from=2-1, to=2-2]
\end{tikzcd}
\quad \iff \quad
\begin{tikzcd}
	{\Q[W(H)]} & {Z^0X(H)} \\
	0 & {Z^0Y(H)}
	\arrow["a", from=1-1, to=1-2]
	\arrow[from=1-1, to=2-1]
	\arrow["Z^0f(H)", from=1-2, to=2-2]
	\arrow[dotted, from=2-1, to=1-2]
	\arrow[from=2-1, to=2-2]
\end{tikzcd}
\]
Thus, existence of a lift in the leftmost diagram is equivalent to any element $a\in Z^0X(H)$ belonging to the kernel of $Z^0f(H)$. In other words, $Z^0K(H)=0$. But $Z^0K=H^0K=\ker(H^0f)$, so this means $H^0(f)(H)$ is injective.
\end{itemize}
\smallskip

Now, we prove that $\FW\subset I^\boxslash$. Lifts against $\{S^0(\PH)\to 0\}_{H\in \OG}$ are a particular case of what we just proved. We continue; suppose $f\in \FW$:
\smallskip

\begin{itemize}
\item \underline{Maps in $\FW$ admit lifts against $\{S^n(U_t)\to D^n(U_t)\}_{n\geq 1, t\in T}$:}

Let $n\geq 1$ and $t\in T$. By (\ref{subsub-spheres}), commutativity of both of these diagrams is equivalent,
\[\begin{tikzcd}
	{S^n(U_t)} & X \\
	{D^n(V_t)} & Y
	\arrow["\alpha", from=1-1, to=1-2]
	\arrow[from=1-1, to=2-1]
	\arrow["f", from=1-2, to=2-2]
	\arrow["\beta"', from=2-1, to=2-2]
\end{tikzcd} 
\quad \iff \quad%
\begin{tikzcd}
	{U_t} & {Z^nX} \\
	{V_t} & {Z^nY}
	\arrow["a", from=1-1, to=1-2]
	\arrow["{\delta_t}"', from=1-1, to=2-1]
	\arrow["f", from=1-2, to=2-2]
	\arrow["d \circ b"', from=2-1, to=2-2]
\end{tikzcd}\]
and a lift in the diagram on the left is equivalent to a morphism $h:V_t\to X^{n-1}$ such that the diagrams (\ref{eq-spheredisk-lift}) commute, i.e. $fh=b:V_t\to Y^{n-1}$ and $dh\delta_t=a$. We will now define such a morphism. Since $f^{n-1}$ is surjective with injective kernel, the short exact sequence \begin{equation} \label{eq-znkd} \tag{$\sharp$} 0\to K^{n-1}\hookrightarrow X^{n-1} \stackrel{f}{\to} Y^{n-1}\to 0\end{equation} splits: let $s:Y^{n-1}\to X^{n-1}$ be such that $fs=\id$. Consider the morphism $a-dsb\delta_t:U_t\to Z^nX$. Then
\[f(a-dsb\delta_t)=fa-fdsb\delta_t=db\delta_t-dfsb\delta_t=0,\]
so $a-dsb\delta_t$ lands in $\ker(Z^nf)=Z^nK$.

We will now use that $Z^nK$ is injective for all $n\geq 0$. We prove this by induction. First, $Z^0K=H^0K=0$ because $K$ is acyclic. Assume $Z^{n-1}K$ is injective for $n\geq 1$. Consider the sequence \begin{equation}\tag{$\ddagger$}\label{eq-znk2} 0\to Z^{n-1}K \hookrightarrow K^{n-1} \stackrel{d}{\to} Z^nK\to 0.\end{equation} It is exact on the right because $K$ is acyclic, so it is a short exact sequence beginning with an injective. Thus, it splits, so $Z^nK$ is a summand of $K^{n-1}$ which is injective, so $Z^nK$ is injective.

Therefore, there exists an extension $e$ making the following diagram commute:
\[\begin{tikzcd}
	{U_t} & {V_t} \\
	{Z^nK}
	\arrow["{\delta_t}", from=1-1, to=1-2]
	\arrow["{a-dsb\delta_t}"', from=1-1, to=2-1]
	\arrow["e", dotted, from=1-2, to=2-1]
\end{tikzcd}\]
Since (\ref{eq-znk2}) is exact, we can choose a $\sigma:Z^nK \to K^{n-1}$ such that $d\sigma=\id$. We define $h:V_t \to X^{n-1}$ as $h=\sigma e+sb$. Since $f\sigma e=0$, then $fh=fsb=b$, and
\[dh\delta_t=d\sigma e \delta_t + dsb\delta_t = (a-dsb \delta_t) +dsb\delta_t=a.\]
Thus, $h$ is the desired morphism. \smallskip

\item \underline{Maps in $\FW$ admit lifts against $\{S^n(\PH)\to D^n(\PH)\}_{n\geq 1, H\in \OG}$:}

Let $n\geq 1$ and $H\in \OG$. By (\ref{subsub-idPH}), a commutative diagram
\[\begin{tikzcd}
	{S^n(\PH)} & X \\
	{D^n(\PH)} & Y
	\arrow["\alpha", from=1-1, to=1-2]
	\arrow[from=1-1, to=2-1]
	\arrow["f", from=1-2, to=2-2]
	\arrow["\beta"', from=2-1, to=2-2]
\end{tikzcd}\]
is equivalent to elements $a\in Z^nX(H)$ and $b\in Y^{n-1}(H)$ such that $f(a)=db$. A lift in the diagram is equivalent to an element $h\in X^{n-1}(H)$ such that $f(h)=b$ and $dh=a$, which we now define.

Since $f$ is surjective, there exists an $x\in X^{n-1}(H)$ such that $f(x)=b$. Now, $a-dx\in Z^nK(H)=\ker(Z^nf(H))$:
\[f(a-dx)=fa-fdx=db-dfx=db-db=0.\]
But since $K$ is acyclic, then $H^nK=0$ so $a-dx=dx'$ for some $x'\in K^{n-1}(H)$. Define $h\in X^{n-1}(H)$ by $h=x'+x$. Then $f(h)=f(x)=b$ and $dh=a-dx+dx=a$, so $h$ is the desired element.
\end{itemize}
\smallskip
We now prove that $I^\boxslash \subset \FW$. Assume $f$ is in $I^\boxslash$.
\smallskip
\begin{itemize}
    \item \underline{$Z^nf:Z^nX\to Z^nY$ epimorphism for all $n\geq 0$}

Let $n\geq 0$, $H\in \OG$ and $b\in Z^nY(H)$, so $db=0$. By (\ref{subsub-idPH}), the choice of elements $0\in Z^{n+1}X(H)$ and $b$ with $f(0)=db$ determines a commutative diagram
\begin{equation} \tag{$\flat$}\label{eq-Z^nf-epi}\begin{tikzcd}
	{S^{n+1}\PH} & X \\
	{D^{n+1}\PH} & Y
	\arrow["0", from=1-1, to=1-2]
	\arrow[from=1-1, to=2-1]
	\arrow["f", from=1-2, to=2-2]
	\arrow["\beta"', from=2-1, to=2-2]
\end{tikzcd}\end{equation}
which admits a lift by hypothesis. This lift determines an element $h\in X^n(H)$ such that $fh=b$ and $dh=0$, so $h\in Z^nX(H)$ is the desired inverse image of $b$ under $Z^nf$.

\smallskip
\item \underline{$f$ is an epimorphism:}

Let $n\geq 0$, $H\in \OG$ and $b\in Y^n(H)$, so $db\in Z^{n+1}Y(H)$. Since $Z^{n+1}f:Z^{n+1}X(H)\to Z^{n+1}Y(H)$ is surjective by the previous point, there exists an $a\in Z^{n+1}X(H)$ such that $fa=db$. This determines a commutative diagram as (\ref{eq-Z^nf-epi}) with $\alpha$ instead of 0. By hypothesis, there exists a lift, which amounts to an element $h\in X^n(H)$ such that $fh=b$ (and $dh=a$), so $h$ is the desired inverse image of $b$ under $f$. 

\smallskip
\item \underline{$f$ is a quasi-isomorphism:}

Since $Z^nf$ is an epimorphism for all $n\geq 0$, then the same is true of $H^nf$. It remains to show $H^nf$ is a monomorphism for all $n\geq 0$. Let $H\in \OG$.

Let $n\geq 0$. We first prove $H^{n+1}f:H^{n+1}X(H)\to H^{n+1}Y(H)$ is injective. Let $[a]\in H^{n+1}X(H)$ such that $H^{n+1}f[a]=0$, i.e. $fa=db$ for some $b\in Y^n(H)$.  As before, this determines a commutative diagram as in (\ref{eq-Z^nf-epi}) with $\alpha$ instead of 0. By hypothesis, there exists a lift, which amounts to an element $h\in X^n(H)$ such that $dh=a$ (and $fh=b$), which means $[a]=0$.

We now show $H^0f(H)=Z^0f(H):Z^0X(H)\to Z^0Y(H)$ is injective.  Let $a\in Z^0X(H)$ such that $f(a)=0$. Thus, the diagram on the left commutes, and by the $S^0\dashv Z^0$ adjunction this is equivalent to commutativity of the second diagram: %
\[
\begin{tikzcd}
	{\Q[W(H)]} & {Z^0X(H)} \\
	0 & {Z^0Y(H)}
	\arrow["a", from=1-1, to=1-2]
	\arrow[from=1-1, to=2-1]
	\arrow["f", from=1-2, to=2-2]
	\arrow[dotted, from=2-1, to=1-2]
	\arrow[from=2-1, to=2-2]
\end{tikzcd}
\quad \iff \quad
\begin{tikzcd}
	{S^0(\PH)} & X \\
	0 & Y
	\arrow[from=1-1, to=1-2]
	\arrow[from=1-1, to=2-1]
	\arrow["f", from=1-2, to=2-2]
   	\arrow[dotted, from=2-1, to=1-2]
	\arrow[from=2-1, to=2-2]
\end{tikzcd}
\]
A lift exists by hypothesis: this means $a=0$, as we wanted.
\smallskip

\item \underline{$Z^nK$ is injective for all $n\geq 1$:}

Let $n\geq 1$, $t\in T$ and $\phi:U_t\to Z^nK\subset Z^nX$. We want to find an extension as in the following diagram:
\[\begin{tikzcd}
	{U_t} & {V_t} \\
	{Z^nK}
	\arrow["{\delta_t}", from=1-1, to=1-2]
	\arrow["\phi"', from=1-1, to=2-1]
	\arrow["e", dotted, from=1-2, to=2-1]
\end{tikzcd}\]
By (\ref{subsub-spheres}), commutativity of the following two diagrams is equivalent:
\[\begin{tikzcd}
	{S^n(U_t)} & X \\
	{D^n(V_t)} & Y
	\arrow[from=1-1, to=1-2]
	\arrow[from=1-1, to=2-1]
	\arrow["f", from=1-2, to=2-2]
	\arrow["0"', from=2-1, to=2-2]
\end{tikzcd}
\quad \iff \quad%
\begin{tikzcd}
	{U_t} & {Z^nX} \\
	{V_t} & {Z^nY}
	\arrow["{\phi}", from=1-1, to=1-2] %
	\arrow["{\delta_t}"', from=1-1, to=2-1]
	\arrow["f", from=1-2, to=2-2]
	\arrow["0"', from=2-1, to=2-2]
\end{tikzcd}\]
By hypothesis, there exists a lift in the diagram on the left, which amounts to a morphism $h:V_t\to X^{n-1}$ such that $fh=0$ and $dh\delta_t=\phi$. The first condition means $h$ lands in $K^{n-1}$, so we can define $e:V_t\to Z^nK$ to be $dh$, and it satisfies $e\delta_t=\phi$: it is our desired extension.

\smallskip
\item \underline{$K^n$ is injective for all $n\geq 0$:}

Let $n\geq 0$. We have already proven that $f$ is a quasi-isomorphism and an epimorphism. In particular, $K$ is acyclic, so we have a short exact sequence
\[0\to Z^nK \hookrightarrow K^n \stackrel{d}{\to} Z^{n+1}K \to 0.\]
If $n\geq 1$, then the sequence splits because $Z^nK$ is injective by the previous point. Thus, $K^n$ is isomorphic to $Z^nK \oplus Z^{n+1}K$, a direct sum of two injectives, so $K^n$ is injective. For $n=0$, we have $Z^0K=H^0K=0$: the short exact sequence gives $K^0\cong Z^1K$ which is injective. \qedhere
\end{itemize}
\end{proof}

\begin{remark}
    In the proof of \cref{thm-IJ}, we appeal to \cite{Bousfield--CosimplicialResolutions} in order to establish the existence of the model structure. Our work consists in checking it is cofibrantly generated by $I$ and $J$. In fact, with a bit more work one can bypass the use of \cite{Bousfield--CosimplicialResolutions} entirely. Indeed, from the characterizations of $I$ and $J$ it follows immediately that $I^\boxslash=W\cap J^\boxslash$, where $W$ is the class of weak equivalences.  
    Thus, we would need to check only one more hypothesis from Kan's criterion on the existence of cofibrantly generated model structures \cite[Thm 2.1.19]{Hovey-ModelCategories}%
    , namely, that $J$-cell complexes are weak equivalences. %
    Afterwards, we would need to check that ${}^\boxslash(I^\boxslash)$  is the class of monomorphisms in positive degrees.
\end{remark}

The proof of \cref{thm-IJ} uses elements in order to make it more readable, but they are not essential. An element-free proof would consistently exploit that $\{\PH\}_{H\in \OG}$ is a set of generators; observe we never used they are projective. Thus, the proof of \cref{thm-IJ} yields:

\begin{theorem} \label{thm-grothendieck}
Let $\A$ be a Grothendieck abelian category with generators $\{P_H\}_{H\in \O}$ indexed by a set $\O$. Let $\{U_t\to V_t\}_{t\in T}$ be a Baer set in $\A$. The \emph{injective} model structure on $\Chz(\A)$ exists: its weak equivalences are the quasi-isomorphisms, the fibrations are the epimorphisms with degreewise injective kernel, and the cofibrations are the monomorphisms in positive degrees. It is cofibrantly generated, with generating acyclic cofibrations given by %
\[J=\left\{0\to D^n(P_H)\right\}_{\substack{n\geq 1 \\ H\in \O}} \quad \cup \quad \left\{D^n(U_t)\to D^n(V_t)\right\}_{\substack{n\geq 1 \\ t\in T}}\]
        and generating cofibrations given by 
\[I=\left\{S^n(P_H) \to D^n(P_H)\right\}_{\substack{n\geq 1 \\ H \in \O}} \quad \cup \quad \left\{S^0(P_H) \to 0 \right\}_{H\in \O} \quad \cup \quad \left\{S^n(U_t) \to D^n(V_t)\right\}_{\substack{n\geq 1 \\ t\in T}}.\]
Every object is cofibrant, and the fibrant objects are the complexes which are degreewise injective.
\end{theorem}

\begin{remark}
As we mentioned in (\ref{subsub-baer}), in a Grothendieck abelian category the inclusions of subobjects of a generator $P$ into $P$ form a Baer set. Using this Baer set, the sets above simplify as follows: \[J=\left\{D^n(U) \to D^n(P)\right\}_{\substack{n\geq 1 \\ U\in \Sub(P)}}\] %
and
\[I=\left\{S^n(U) \to D^n(U)\right\}_{\substack{n\geq 1 \\ U\in \Sub(P)}}\quad \cup \quad \{S^0(P)\to 0\}.\]
This follows simply from the decomposition $\Sub(P)=\{0\} \cup (\Sub(P)\setminus \{0\})$ for $J$, and from the decomposition $\Sub(P)=\{P\}\cup (\Sub(P)\setminus \{P\})$ for $I$.
\end{remark}

\subsection{The \texorpdfstring{$r(\inj)$}{r(inj)} model structure on \texorpdfstring{$\CDGA^\OG$}{Fun(OG,CDGA)}} \label{sect-rinj}

Let $F:\Chz \rightleftarrows \CDGA:U$ be the free--forgetful adjunction. Let \[F_*:(\Chz)^\OG \rightleftarrows \CDGA^\OG:U_*\] be the induced adjunction by post-composition. Consider the injective model structure on $(\Chz)^\OG \cong \Chz(\Vect^\OG)$ from \cref{thm-IJ} (see also \cref{rmk-inj-model-st}). We call the right-transferred model structure on $\CDGA^\OG$ along $U_*$ the \emph{$r(\inj)$ model structure} on $\CDGA^\OG$.

We will now prove this structure exists and is cofibrantly generated. The generating cofibrations (resp. generating acyclic cofibrations) will be given by $F_*I$ (resp. $F_*J$), with $I$ and $J$ as in (\ref{subsub-IJ}). Let us call these sets $\I$ and $\J$. We introduce some notation. Let $\D^n_H:=F_*D^n(\PH)$, $\D^nW:=F_*D^n(W)$, for $n\geq 1$, $H\in \OG$, $W\in \Vect^\OG$, and we similarly define $\S^n$ for $n\geq 0$ using $S^n$ instead of $D^n$. Thus,
\[\J=\left\{\Q\to \D^n_H\right\}_{\substack{n\geq 1 \\ H\in \OG}} \quad \cup \quad \left\{\D^nU_t\to \D^nV_t\right\}_{\substack{n\geq 1 \\ t\in T}},\]

\[\I=\left\{\S^n_H \to \D^n_H\right\}_{\substack{n\geq 1 \\ H \in \OG}} \quad \cup \quad \left\{\S^0_H \to \Q \right\}_{H\in \OG} \quad \cup \quad \left\{\S^nU_t \to \D^nV_t\right\}_{\substack{n\geq 1 \\ t\in T}}.\]
Here, $\Q\in \CDGA^\OG$ is component-wise the unit of $\CDGA$.

\begin{theorem} \label{thm-IJ-cdga}
    The $r(\inj)$ model structure on $\CDGA^\OG$ exists. Its weak equivalences are the component-wise quasi-isomorphisms and its fibrations are the degreewise component-wise surjections with degreewise injective kernel as an object of $\Vect^\OG$. It is cofibrantly generated by the sets $\I$ and $\J$ defined above. The fibrant objects are the systems of cdgas which are degreewise injective as objects of $\Vect^\OG$.
    \begin{proof}
We will use Kan's standard criterion from \cite[Thm. 11.3.2]{Hirschhorn} on right-transferred model structures. Since $\CDGA^\OG$ is locally presentable (by \cref{lem-cdga-combinatorial} and \cite[1.54]{adamek-rosicky}), we need only check that %
the relative $\J$-cell complexes are quasi-isomorphisms. In fact, we prove they are acyclic cofibrations in $\Chz(\Vect^\OG)$. We do this in steps. We use that pushouts in $\CDGA^\OG$ are computed component-wise. Moreover, pushouts in $\CDGA$ are relative tensor products. %
\begin{itemize}
\item \underline{$\Q\to\D^nW$ are acyclic cofibrations in $\Chz$,  for $W\in \Vect$:}

Here, by $\D^n$ ($n\geq 1$) we mean the composition $\Vect \stackrel{D^n}{\to} \Chz \stackrel{F}{\to} \CDGA$. This is a slight abuse of notation: these are the non-equivariant versions of $D^n$ and $\D^n$ defined before.  The map $\Q\to \D^nW$ in question is the unique map in $\CDGA$, namely the unit. We now consider it as a map in $\Chz$. It is the inclusion of 0 in positive degrees, so it is a cofibration. %
Let us prove it is a quasi-isomorphism.

Taking a basis for $W$, we may write $W\cong \Q^{\oplus I}$ for some set $I$. Since $\D^n$ preserves colimits, we have $\D^nW\cong \D^n(\Q)^{\otimes I}$. (If $I$ is infinite, this means the filtered colimit of the tensor product of $\D^n\Q$ over the finite subsets of $I$.) %
By the Künneth theorem, %
(and, if $I$ is infinite, because cohomology commutes with filtered colimits), it suffices to prove $\Q \to \D^n\Q$ is a quasi-isomorphism. This is standard, see e.g. \cite[Proof of Lem. 6.2.6]{Fresse-Teichmuller2}.

\smallskip
\item \underline{Pushouts of maps $\Q\to \D^n_H$ are acyclic cofibrations in $\Chz(\Vect^\OG)$:}

Let $A\in \CDGA^\OG$, $H\in \OG$ and $n\geq 1$. There is a unique map $\Q\to A$ in $\CDGA^\OG$, namely component-wise the unit. We take its pushout along $\Q\to \D^n_H$ and evaluate on an $H'\in \OG$, to get the following pushout diagram in $\CDGA$:
\[\begin{tikzcd}
	\Q & {A(H')} \\
	{\D^n_H(H')} & {A(H')\otimes\D^n_H(H')}
	\arrow[from=1-1, to=1-2]
	\arrow["\lambda"', from=1-1, to=2-1]
	\arrow[from=1-2, to=2-2]
	\arrow[from=2-1, to=2-2]
	\arrow["\lrcorner"{anchor=center, pos=0.125, rotate=180}, draw=none, from=2-2, to=1-1]
\end{tikzcd}\]
Consider it as a diagram in $\Chz$. %
The left vertical map, call it $\lambda$, is an acyclic cofibration by the previous point. The map on the right is $A(H')\cong A(H')\otimes \Q \stackrel{\id \otimes \lambda}{\to} A(H') \otimes \D^n_H(H')$, which is injective %
and a quasi-isomorphism by the Künneth theorem.

\smallskip
\item \underline{Pushouts of maps $\D^nU_t\to \D^nV_t$ are acyclic cofibrations in $\Chz(\Vect^\OG)$:}

Let $H\in \OG$, $n\geq 1$, $t\in T$ and $\phi:\D^nU_t\to A$ be a map in $\CDGA^\OG$. Take the cokernel $C$ of the subspace inclusion $U_t(H)\hookrightarrow V_t(H)$. %
Thus, we have an isomorphism $V_t(H)\cong U_t(H) \oplus C$ in the undercategory $\Vect_{U_t(H)/}$. Apply the left adjoint $\D^n:\Vect \to \CDGA$ to obtain an isomorphism $\D^nV_t(H) \cong \D^nU_t(H) \otimes \D^n C$ in $\CDGA_{\D^nU_t(H)/}$. %
Thus, we can identify the pushout of $\phi(H)$:
\[\begin{tikzcd}
	{\D^nU_t(H)} & {A(H)} & {A(H)\otimes \Q} \\
	{\D^nV_t(H)} & {A(H)\otimes_{\D^nU_t(H)}(\D^nU_t(H) \otimes \D^nC)} & {A(H)\otimes \D^n C}
	\arrow[""{name=0, anchor=center, inner sep=0}, "{\phi(H)}", from=1-1, to=1-2]
	\arrow[from=1-1, to=2-1]
	\arrow["\cong", from=1-2, to=1-3]
	\arrow[from=1-2, to=2-2]
	\arrow["{\id \otimes \lambda}", from=1-3, to=2-3]
	\arrow[from=2-1, to=2-2]
	\arrow["\cong"', from=2-2, to=2-3]
	\arrow["\lrcorner"{anchor=center, pos=0.125, rotate=180}, draw=none, from=2-2, to=0]
\end{tikzcd}\]
As before, the unit $\lambda:\Q\to \D^nC$ is an acyclic cofibration in $\Chz(\Vect^\OG)$ by the first point, so $\id\otimes \lambda$ is an injection as well and a quasi-isomorphism by the Künneth theorem.

\smallskip
\item \underline{Transfinite compositions of pushouts of maps in $\J$ are quasi-isomorphisms:}

We just proved that pushouts of maps in $\J$ are acyclic cofibrations in $\Chz(\Vect^\OG)$. But transfinite compositions of acyclic cofibrations in a model category are acyclic cofibrations, so in particular, weak equivalences. Since the functor $U_*:\CDGA^\OG\to (\Chz)^\OG$ preserves sequential %
colimits, this finishes the proof. %
\qedhere
        \end{itemize}
    \end{proof}
\end{theorem}

\subsection{The \texorpdfstring{$r(\inj)$ model structure on $(\CDGA^1)^\OG$}{r(inj) model structure on Fun(OG,CDGA1)}} \label{sect-rinj-connected}

Let $\CDGA^1$ denote the full subcategory of $\CDGA$ on the connected cdgas (see \cref{def-cdga-conventions}). Our aim in this section is to get a model structure on $(\CDGA^1)^\OG$ analogous to the $r(\inj)$ model structure on $\CDGA^\OG$.

A connected cdga $A$ is augmented in a unique way, and its differential $d^0$ is zero. The augmentation ideal of $A$ (i.e. the kernel of the augmentation), which is simply its positive part, is a non-unital cdga concentrated in positive degrees, and this defines an equivalence of categories: the essential inverse is $\Q\oplus -$. We can reinterpret this category of non-unital cdgas as follows:

Let $\Chun$ denote the category of positively graded cochain complexes of rational vector spaces. It is isomorphic to $\Chz$ via a shift.  Consider the functor $\Chun\to \Chun$, $V\mapsto\bigoplus_{k\geq 1}(V^{\otimes k})_{\Sigma_k}$. It defines a monad on $\Chun$ such that the category of algebras over it %
is equivalent to the category of non-unital cdgas concentrated in positive degrees, and therefore to $\CDGA^1$. Henceforth, we will make this identification without further notice. Thus, we have a free--forgetful adjunction
\[\begin{tikzcd}
	\Chun & {\CDGA^1.}
	\arrow[""{name=0, anchor=center, inner sep=0}, "F", from=1-1, to=1-2]
	\arrow[""{name=1, anchor=center, inner sep=0}, "(-)^{\geq1}", shift left=3, from=1-2, to=1-1]
	\arrow["\dashv"{anchor=center, rotate=-90}, draw=none, from=0, to=1]
\end{tikzcd}\]
Here $F$ is in fact the restriction of the free functor $\Chz\to \CDGA$ to its subcategory $\Chun$, and $(-)^{\geq 1}$ is the %
positive part of the connected cdga. 

This induces an adjunction %
\[\begin{tikzcd}
	{(\Chun)^\OG} & {(\CDGA^1)^\OG.}
	\arrow[""{name=0, anchor=center, inner sep=0}, "{F_*}", from=1-1, to=1-2]
	\arrow[""{name=1, anchor=center, inner sep=0}, "{(-)^{\geq1}_*}", shift left=3, from=1-2, to=1-1]
	\arrow["\dashv"{anchor=center, rotate=-90}, draw=none, from=0, to=1]
\end{tikzcd}\] Using the isomorphisms $(\Chun)^\OG \cong \Chun(\Vect^\OG) \cong \Chz(\Vect^\OG)$, we endow $(\Chun)^\OG$ with the model structure from \cref{thm-IJ}. 
Our aim is to right-transfer it to $(\CDGA^1)^\OG$, just as in \cref{sect-rinj}.

For clarity, let us first make the model structure on $(\Chun)^\OG$ explicit. The weak equivalences are the quasi-isomorphisms, the cofibrations are the injections component-wise in degrees $\geq 2$, and the fibrations are the degreewise component-wise surjections with degreewise injective kernel. %
It is cofibrantly generated by the following sets $J^{\geq1}$ and $I^{\geq1}$:

\[J^{\geq1}=\left\{0\to D^n(\PH)\right\}_{\substack{n\geq 2 \\ H\in \OG}} \quad \cup \quad \left\{D^n(U_t)\to D^n(V_t)\right\}_{\substack{n\geq 2 \\ t\in T}}.\]
\[I^{\geq1}=\left\{S^n(\PH) \to D^n(\PH)\right\}_{\substack{n\geq 2 \\ H \in \OG}} \quad \cup \quad \left\{S^1(\PH) \to 0 \right\}_{H\in \OG} \quad \cup \quad \left\{S^n(U_t) \to D^n(V_t)\right\}_{\substack{n\geq 2 \\ t\in T}}.\]

We can now establish the desired model structure. We use the notations from \cref{sect-rinj} for the generating (acyclic) cofibrations.

\begin{theorem} \label{thm-rinj-connected}
The right-transferred model structure on $(\CDGA^1)^\OG$ along $(-)^{\geq1}_*$ exists: we call it the \emph{$r(\inj)$ model structure}. Its weak equivalences are the component-wise quasi-isomorphisms and its fibrations are the degreewise component-wise surjections with degreewise injective kernel as an object of $\Vect^\OG$. It is cofibrantly generated, with set of generating cofibrations $\I^1$ and set of generating acyclic cofibrations $\J^1$ given as follows:
\[\J^1=\left\{\Q\to \D^n_H\right\}_{\substack{n\geq 2 \\ H\in \OG}} \quad \cup \quad \left\{\D^nU_t\to \D^nV_t\right\}_{\substack{n\geq 2 \\ t\in T}},\]

\[\I^1=\left\{\S^n_H \to \D^n_H\right\}_{\substack{n\geq 2 \\ H \in \OG}} \quad \cup \quad \left\{\S^1_H \to \Q \right\}_{H\in \OG} \quad \cup \quad \left\{\S^nU_t \to \D^nV_t\right\}_{\substack{n\geq 2 \\ t\in T}}.\] %

The fibrant objects are the systems of connected cdgas which are degreewise injective as objects of $\Vect^\OG$.
\begin{proof} We will again apply Kan's criterion from \cite[Thm. 11.3.2]{Hirschhorn} on right-transferred model structures.

The category $\CDGA^1$ is locally presentable: the proof is the same as in \cref{lem-cdga-combinatorial}. %

As in any category of algebras over a monad, limits are created by the forgetful functor $(-)^{\geq1}_*$. (This means in particular that products in $\CDGA^1$ are the fiber products over the augmentations: they do not coincide with the products in $\CDGA$.) 
On the other hand, coproducts and coequalizers of connected cdgas stay connected, so $\CDGA^1$ is closed under colimits in $\CDGA$. Thus, $\CDGA^1$ is complete and cocomplete, so $(\CDGA^1)^\OG$ is also complete and cocomplete.

Applying $F_*$ to the sets $J^{\geq1}$ and $I^{\geq1}$ gives exactly the sets $\J^1$ and $\I^1$. Thus, we only need to check that relative $\J^1$-cell complexes are quasi-isomorphisms. But $\J^1\subset \J$ and colimits in $(\CDGA^1)^\OG$ are computed in $\CDGA^\OG$, so this follows from relative $\J$-cell complexes being quasi-isomorphisms. Therefore, Kan's criterion applies.

Finally, the classes of weak equivalences and fibrations are exactly as described. For fibrations, this is obvious: a map $f:X\to Y$ in $(\CDGA^1)^\OG$ is a fibration if and only if $f^{\geq1}:X^{\geq1}\to Y^{\geq1}$ is a fibration in $(\Chun)^\OG$. This means that $f^{\geq1}$ is degreewise component-wise surjective with degreewise injective kernel. But those same conditions in degree zero are automatic, since the cdgas are connected. 
For weak equivalences, it is similar: a map $f$ in $(\CDGA^1)^\OG$ is a weak equivalence if and only if $f^{\geq1}$ is a quasi-isomorphism, which is equivalent to $f$ being a quasi-isomorphism since the cdgas are connected. %
\end{proof}
\end{theorem}

\begin{remark}  \label{rmk-model-cdga1} When taking $G=\{e\}$, the $r(\inj)$ model structure on $(\CDGA^1)^\OG$ amounts to a model structure on $\CDGA^1$ whose weak equivalences are the quasi-isomorphisms, the fibrations are the degreewise surjections, and the generating (acyclic) cofibrations are obtained by applying the functor $F$ to the generating (acyclic) cofibrations of $\Chun$. 
\end{remark}

\begin{remark} \label{rmk-model-cdga1-via-aug} We indicate another description of the model structure in $(\CDGA^1)^\OG$. First, observe that $\CDGA^1$ is a reflective subcategory of the category of augmented cdgas $\CDGA_{/\Q}$. Indeed, it is a full subcategory since morphisms of connected cdgas automatically commute with the unique augmentations. The reflector $L:\CDGA_{/\Q}\to \CDGA^1$ is given by taking the quotient by the dg ideal generated by the kernel of the augmentation in degree 0. Thus, $L$ does not alter the cdga in degrees $\geq 2$.

Endow the overcategory $(\CDGA_{r(\inj)}^\OG)_{/\Q}$ with the %
model structure from \cref{prop-overunder}. 
It is cofibrantly generated, with generating (acyclic) cofibrations given by the generating (acyclic) cofibrations of $\CDGA^\OG_{r(\inj)}$, where we endow their (co)domains with all possible augmentations turning them into morphisms that respect the augmentations.

There is an isomorphism of categories $(\CDGA^\OG)_{/\Q}\cong (\CDGA_{/\Q})^\OG$, %
so we can drop the parentheses. Thus, we have an induced adjunction by post-composition, where $i:\CDGA^1\hookrightarrow \CDGA_{/\Q}$ denotes the inclusion:
\[\begin{tikzcd}
	{\CDGA_{/\Q}^\OG} & {(\CDGA^1)^\OG}
	\arrow[""{name=0, anchor=center, inner sep=0}, "{L_*}", from=1-1, to=1-2]
	\arrow[""{name=1, anchor=center, inner sep=0}, "{i_*}", shift left=3, from=1-2, to=1-1]
	\arrow["\dashv"{anchor=center, rotate=-90}, draw=none, from=0, to=1]
\end{tikzcd}\]
    We now observe that the $r(\inj)$ model structure on $(\CDGA^1)^\OG$ from \cref{thm-rinj-connected} is also the right-transferred model structure along this $i_*$. Indeed, in this model structure a morphism $f$ in $(\CDGA^1)^\OG$ is a fibration or a weak equivalence if and only if it is so in $\CDGA^\OG_{r(\inj)}$, which is precisely the characterization of the fibrations and weak equivalences in \cref{thm-rinj-connected}.

    In fact, one can alternatively establish the existence of this model structure using Kan's transfer theorem for the functor $i_*$. That is a similarly uncomplicated proof, but there is a disadvantage to this approach: it produces sets of generating (acyclic) cofibrations which are larger, due to the fact of the generating (acyclic) cofibrations in the augmented category having to take into account all the possible augmentations.
\end{remark}

\begin{remark} Even when restricted to reduced simplicial sets, the functor $\APL$ from \cref{thm-fund} need not produce connected cdgas. We can still make a connection between $(\CDGA^1)^\OG$ and $\sSet_*^\OGop$, albeit in the form of a zig-zag. It is not hard to prove that the Quillen adjunction $L_*\dashv i_*$ from the previous remark is a partial Quillen equivalence on the systems of cohomologically connected augmented cdgas weakly of finite type and their connected counterparts. We get the following zig-zag of partial Quillen equivalences, where the first one was explained in \cref{rmk-pointed-variant}:
\[\begin{tikzcd}
	{\sSet_*^\OGop} & {(\CDGA_{/\Q}^\OG)_{r(\inj)}^\op} & {((\CDGA^1)^\OG)_{r(\inj)}^\op.}
	\arrow[""{name=0, anchor=center, inner sep=0}, "{(\APL)_*}", from=1-1, to=1-2]
	\arrow[""{name=1, anchor=center, inner sep=0}, "{\langle - \rangle_*}", shift left=3, from=1-2, to=1-1]
	\arrow[""{name=2, anchor=center, inner sep=0}, "{L_*^\op}"', shift right=3, from=1-2, to=1-3]
	\arrow[""{name=3, anchor=center, inner sep=0}, "{i_*^\op}"', from=1-3, to=1-2]
	\arrow["\dashv"{anchor=center, rotate=-90}, draw=none, from=0, to=1]
	\arrow["\dashv"{anchor=center, rotate=-90}, draw=none, from=3, to=2]
\end{tikzcd}\]
\end{remark}

\section{Remarks on other models in the literature} \label{sec-other-models}

In this final section, we make some comments on other models for equivariant rational homotopy theory.

\subsection{Relation to Scull's work} \label{sec-scull}

We now explain the relation between our work and \cite{Scull-2008}, pointing out some inaccuracies in that article that we resolve.

\begin{enumerate}
    \item In \cite[Thm 3.1]{Scull-2008}, Scull defines a model structure on $\Chz(\Vect^\OG)$ that would coincide with the one in our \cref{thm-model-ChzVectOG}, since it would have the same weak equivalences and cofibrations. However, her fibrations  have injective kernel only in positive degrees. %
This is not correct, and we can give a counterexample. Let $A\in \Vect^\OG$ be a non-injective object, witnessed by a monomorphism $i:V\to W$ and a map $\phi:V\to A$ with no extension to $W$.

(For a concrete example, take $G=\Z/2\Z$. Then $\Vect^\OG$ is the category of representations of the quiver
\[\begin{tikzcd}
	{V_1} & {V_2}
	\arrow["a", from=1-1, to=1-1, loop, in=150, out=210, distance=5mm]
	\arrow["b", from=1-1, to=1-2]
\end{tikzcd}\]
subject to the relations $a^2=\id$, $ba=b$. 
Take a representation $V$ defined by $V_1=0$, $V_2=\Q$, $a=b=0$ and a representation $V'$ defined by $V_1'=V_2'=\Q$, $a'=b'=\id_\Q$. Let $i:V\to V'$ be the inclusion. Then $\id_V:V\to V$ does not extend to $V'$ along $i$: if there was an extension $h=(h_1,h_2):V'\to V$, then a component-wise examination shows $h_1=0$ and $h_2=\id_\Q$. On the other hand, since $h$ is a morphism of representations it needs to satisfy $b\circ h_1=h_2\circ b'$. This means $h_2=0$, a contradiction. Therefore, $V$ is not injective.) %

In other words, there is no lift in the following diagram in $\Vect^\OG$:
\[\begin{tikzcd}
	V & A \\
	W & 0.
	\arrow["\phi", from=1-1, to=1-2]
	\arrow["i"', from=1-1, to=2-1]
	\arrow[from=1-2, to=2-2]
	\arrow[from=2-1, to=2-2]
\end{tikzcd}
\]
Since $A=\ev_0(S^0(A))$, a lift in that diagram is equivalent to a lift in the following transpose diagram in $\Chz(\Vect^\OG)$ under the $D^1\dashv \ev_0$ adjunction (see (\ref{subsub-spheres}) for details):
\[\begin{tikzcd}
	{D^1(V)} & {S^0(A)} \\
	{D^1(W)} & 0.
	\arrow[from=1-1, to=1-2]
	\arrow["{D^1(i)}"', from=1-1, to=2-1]
	\arrow[from=1-2, to=2-2]
	\arrow[from=2-1, to=2-2]
\end{tikzcd}\]
Now, the arrow on the left is a monomorphism and a quasi-isomorphism (since disks are acyclic), therefore an acyclic cofibration, and the arrow on the right is an epimorphism whose kernel in positive degrees is zero, therefore injective. So if Scull's description of the fibrations were correct, there should exist a lift in the second (hence first) diagram above, but no lift exists. The solution to this problem is that fibrations have injective kernels in all degrees, as we have in \cref{thm-model-ChzVectOG}.
\item Before Theorem 3.2 in \cite{Scull-2008}, the author makes the convention of having all cdgas $A$ be ``based'', meaning $A^0=\Q$ (in \cref{def-cdga-conventions}, we have called ``connectedness'' the condition of the unit inducing an isomorphism in degree 0). In that theorem, the author considers the category we have denoted $(\CDGA^1)^\OG$ and claims to put a model structure on it, with weak equivalences the component-wise quasi-isomorphisms and fibrations the maps which are component-wise fibrations in $\Chz(\Vect^\OG)$. The proof given in \cite{Scull-2008} is incomplete. Indeed, at the very end of the proof of Ax5 of \cite[Thm. 3.2]{Scull-2008}, it is claimed that ``all maps in $\J$ are weak equivalences, so $i'(f)$ is a weak equivalence'', where $i'(f)$ is a ``transfinite pushout of maps in $\J$''. Here $\J$ is a certain set of maps in the category. But it does not follow that relative cell complexes of maps in $\J$ are weak equivalences from all maps in $\J$ being weak equivalences, so the proof is incomplete. We construct this model structure in \cref{thm-rinj-connected} by different means. In the course of the proof, we do prove that relative cell complexes of the putative generating acyclic cofibrations are quasi-isomorphisms. 
\item In \cite[Def. 3.4, 3.5]{Scull-2008}, the author defines sets of maps $\I$ and $\J$ which are to be the generating (acyclic) cofibrations of $(\CDGA^1)^\OG$ \cite[Prop. 3.6, 3.7]{Scull-2008}.
In our terminology from \cref{sect-rinj-connected}, her $\J$ is as our $\J^1$ (we now fix the Baer set to be the one from \cite[Lem. 6.4]{Scull-2008}: in particular, $G$ is now finite) but with the indexing sets extended to $n\geq 1$. This cannot be correct, since $\D^1$ on a system of non-zero vector spaces gives a non-connected system of cdgas, so those $\I$ and $\J$ are not sets of morphisms of $(\CDGA^1)^\OG$. Restricting to $n\geq 2$ corrects the set of generating acyclic cofibrations, but is not enough to correct the set of generating cofibrations. Indeed, our $\I^1$ also contains the maps $\{\S^1_H\to \Q\}_{H\in \OG}$ and they are essential. To give a simple example, take $G=\{e\}$, so the category reduces to $\CDGA^1$ with the model structure from \cref{rmk-model-cdga1}. Scull's $\I$, corrected to $n\geq 2$, is the set $\{\S^n \to \D^n\}_{n\geq 2}$. %
Let $\Lambda(x)$ be the exterior algebra with zero differentials and $|x|=1$. Consider the augmentation $\epsilon:\Lambda(x)\to \Q$. It is not a quasi-isomorphism, %
but it has the right lifting property against these maps: the zero maps provide lifts.
\item In \cite[Thm. 5.6]{Scull-2008}, there is a claim of a Quillen equivalence between some categories of $\OG$-presheaves of rational pointed simplicial sets and $G$-systems of connected cdgas satisfying additional conditions. That is not a correct statement, as those are not model categories but rather subcategories of model categories. We have given a corrected version of this in \cref{thm-spaces-r(inj)} without hypotheses of pointedness or connectedness. There, we prove there is a \emph{partial} Quillen equivalence. Moreover, we note that our proof is different: it does not go through the theories of minimal models and equivariant Postnikov decompositions, but rather follows from more abstract model categorical considerations.
\item Elaborating on the previous point, we note an additional confusion in that section of the cited article. Before \cite[Thm. 5.1]{Scull-2008}, the standard $\APL$ piecewise linear differential forms functor is defined. However, it is not well-defined, as it takes a pointed simplicial set and is supposed to give a connected cdga. This is not correct: its value on a disconnected, pointed simplicial set is an \emph{augmented} cdga, not a connected one. This problem extends to the rest of the section. One possible solution is to consider pointed simplicial sets and augmented cdgas. %
See \cref{rmk-pointed-variant}.
\item Assume $G$ is finite. The result \cite[Cor. 4.3]{Scull-2008} which establishes that Triantafillou's equivariant minimal models \cite{Triantafillou} give cofibrant approximations is still valid. That result is valid in the categories of systems of cdgas and their connected or augmented variants, all with the respective $r(\inj)$ model structure. 
We do note that in \cite[Sect. 4]{Scull-2008}, the author should restrict to $n\geq 2$ as explained above. 
\end{enumerate}

\subsection{Other models}

\begin{enumerate}
    \item In \cite{Golasinski98}, the author considers  $\CDGA^\OG$ for $G$ finite, but this functor category is endowed with the projective model structure, so this is neither of the model structures that we considered before, i.e. the injective one and the  $r(\inj)$ one. All three model structures are Quillen equivalent, by \cref{prop-comparison-inj-r(inj)} and since projective and injective model structures on the same category are Quillen equivalent when they exist. At the time of the writing of \cite{Golasinski98}, the theory of cofibrantly generated model categories was just being developed. This may explain why the author does not use the now standard fact that the projective structure exists when the target model category is cofibrantly generated \cite[11.6.1]{Hirschhorn}. %

Golasiński's approach to proving the existence of the projective structure applies the dual of a theorem of \cite{edwards-hastings}, using the fact that the category $\OG$ for $G$ finite is an EI-category (i.e. all endomorphisms are automorphisms) which is cofinite. Thus, the set of isomorphism classes of objects is a poset; the condition of cofinality means that each such isomorphism class has only finitely many predecessors. We point out that the check of the dual of the ``condition N'' \cite[Page 45]{edwards-hastings}, which is a hypothesis in Edwards--Hastings' theorem, is not present, and is not immediate to us. %

In \cite[Thm. 2.7]{Golasinski98} the author presents a theorem similar to \cref{cor-sciarappa-cdga} or \cref{thm-spaces-r(inj)} in the case when the group $G$ is Hamiltonian (i.e. every subgroup is normal).
    \item In the introduction to \cite{Mandell_2002}, there is a statement of a theorem similar to \cref{cor-sciarappa-cdga}, but it is only at the homotopy category level. No explicit proof is given, and the group $G$ is taken to be finite.

\item In the article \cite{Martinez--Rivera} there is a coalgebraic approach to the $G$-equivariant homotopy theory of connected %
Kan complexes for a discrete $G$. It builds on the simplicial coalgebraic models of \cite{Raptis--Rivera}, promoting them to the equivariant case. Let $\mathbb{F}$ be an algebraically closed field. In Proposition 4.2, they prove the component-wise simplicial $\mathbb{F}$-chains functor from $\OG$-indexed presheaves of reduced simplicial sets to $\OG$-indexed presheaves of connected simplicial $\mathbb{F}$-coalgebras is homotopically fully faithful, meaning the derived unit at a cofibrant object is a weak equivalence. We make two observations. First, the fact that this is a Quillen adjunction (Thm. 4.1) follows from the more general \cref{lem-adjunction-on-fun} applied to the non-equivariant Quillen adjunction from \cite{Raptis--Rivera} they recall in Thm. 2.9. %
Second, the fact that it is homotopically fully faithful %
(Prop. 4.2) follows from the more general statement in \cref{lem-sciarappa} (cf. \cref{rmk-homotopically-ff}). 

\item Working over a separably closed field of characteristic $p>0$, Bachmann and Burklund \cite[Thm. 1.2 and 1.3]{Bachmann-Burklund} give coalgebraic models for $p$-complete, simply-connected spaces. Their results are phrased $\infty$-categorically, but they should translate into a partial Quillen equivalence which could then be extended to the equivariant case similarly as we did above for cdgas and cdgls. 
\end{enumerate}

\bibliographystyle{alpha}
\bibliography{MyBib}
\end{document}